\documentclass[12pt]{article}
\usepackage[affil-it]{authblk}

\usepackage{amsthm}
\usepackage{amssymb}
\usepackage{amsmath}
\usepackage{amsfonts}
\usepackage{cancel}
\usepackage{times}
\usepackage{bbm}
\usepackage{cite}
\usepackage{color, soul}
\usepackage{hyperref}
\hypersetup{
	colorlinks   = true,
	citecolor    = black,
	linkcolor    = blue
}
\usepackage{tikz}
\usetikzlibrary{arrows.meta,positioning,calc}
\newcommand{\dm}[1]{{\color{red}{[[{\bf Dan:} #1]]}}}

\usepackage[paper=a4paper, top=2.5cm, bottom=2.5cm, left=2.5cm, right=2.5cm]{geometry}

\def\qed{\hfill $\vcenter{\hrule height .3mm
		\hbox {\vrule width .3mm height 2.1mm \kern 2mm \vrule width .3mm
			height 2.1mm} \hrule height .3mm}$ \bigskip}

\def \RR {\mathbb R}

\def \eps {\varepsilon}
\def \vphi {\varphi}

\newtheorem{theorem}{Theorem}
\newtheorem{lem}{Lemma}

\newtheorem{prop}{Proposition}

\theoremstyle{definition}

\theoremstyle{remark}
\newtheorem{remark}{Remark}

\long\def\symbolfootnotetext[#1]#2{\begingroup
	\def\thefootnote{\fnsymbol{footnote}}\footnotetext[#1]{#2}\endgroup}

\begin{document}
	\title{Geometric obstructions to Lipschitz transport between weighted Hessian $\mathrm{CD}(\kappa,\infty)$ manifolds}
	\author{William Dudarov\thanks{Department of Mathematics, University of Washington, Seattle, WA 98195\\Email address: \texttt{wdudarov@uw.edu}} ~~and Dan Mikulincer\thanks{Department of Mathematics, University of Washington, Seattle, WA 98195\\Email address: \texttt{danmiku@uw.edu}}  }
	\date{}
	\maketitle
	\vspace{-1.5cm}
	\begin{abstract}
		We construct a weighted Riemannian manifold $(\mathbb R^2,g,\mu)$ satisfying $\mathrm{CD}(1/2,\infty)$, the curvature-dimension condition, with the following property: if
		$\gamma$ denotes a centered Gaussian measure on $\mathbb R^2$, then there is no Lipschitz map
		$T:(\mathbb R^2,\|\cdot\|) \to (\mathbb R^2,g)$ satisfying $T_\#\gamma=\mu$.
		
		Building on this, we prove a Weyl-type asymptotic law for the eigenvalues of the weighted Laplacian $-\Delta_{g,\mu}$ and show that they are asymptotically negligible when compared to the eigenvalues of $-\Delta_{\gamma}$. These results give strong counterexamples to two questions of E. Milman and complement the recent counterexample of Aryan.
		
		
		
	\end{abstract}
	
	\section{Introduction and main results}
	Let $\gamma^d = \frac{1}{(2\pi)^{\frac{d}{2}}}e^{-\|x\|^2/2}$ be the standard Gaussian on $\RR^d$ and let $\mu = e^{-\vphi(x)}dx$ be another probability measure on $\mathbb{R}^d$. We say that $\mu$ is more log-concave than the Gaussian if it has a convex support and for $\mu$-almost any $x$, $\nabla^2 \vphi(x) \succeq \mathrm{I}_d$. When this
	holds, the celebrated Contraction Theorem of Caffarelli \cite{caffarelli2000monotonicity} states that $\mu$ is a contraction of $\gamma^d$: there exists a map $T:\RR^d\to\RR^d$ with
	$\|\nabla T\|_{\mathrm{op}}\leq 1$ almost everywhere and
	$T_\#\gamma^d=\mu$. Moreover, Caffarelli showed that one may take
	$T=T_{\mathrm{opt}}$, the optimal transport map minimizing
	$
	\int_{\RR^d}\|x-T(x)\|^2\,d\gamma^d(x)
	$
	among all maps transporting \(\gamma^d\) to \(\mu\).
	
	This contraction theorem has proved to be widely influential, and has attracted much attention over the years. In particular, it has inspired a large body of work extending Caffarelli's result in several directions: to broader classes of measures \cite{colombo2017lipschitz,kolesnikov2011mass,carlier2026optimal}, beyond the framework of optimal transport \cite{kim2012generalization,mikulincer2023lipshitz,neeman2022lipschitz,brigati2025heat,shenfeld2024exact}, to tail estimates for the transport maps \cite{colombo2021bounds,fathi2024growth}, and to alternative proofs \cite{chewi2024entropic,kolesnikov2013sobolev,fathi2020proof}. As a recent example for an application, \cite{song2026, HuaSongTudose2026} used this contraction as a key ingredient in proving Talagrand's convexity conjecture. At the same time, most of these results remain essentially Euclidean, relying on the fact that both source and target measures are defined on the same Euclidean metric space $\RR^d$.  While there have been some works on non-Euclidean settings \cite{fathi2024transportation,lopez2025bakry,serres2026contractive,ge2025generalization}, they leave open the following concrete question: 
	\begin{center}
		\emph{Is there an analogue of Caffarelli's theorem for Riemannian manifolds?}
	\end{center}
	
	One such possible analogue, for different Riemannian structures on $\RR^d$, was suggested by E. Milman in \cite{milman2018spectral}. Specifically, in \cite[Conjecture 2*]{milman2018spectral} one considers a Riemannian metric $g$, the weighted metric space $(\RR^d,g,\mu)$, replacing the log-concavity assumption by the so-called \textit{curvature-dimension condition}. To define this condition, let $\mathrm{Vol}_g = \sqrt{\det(g)}$ be the volume measure associated to $g$ and write $\mu = e^{-F(x)}d\mathrm{Vol}_g$. Now, define the generalized Ricci tensor 
	$$\mathrm{Ric}_{g,\mu}  = \mathrm{Ric}_{g} + \nabla^2_gF,$$
	where $\mathrm{Ric}_{g}$ is the Ricci curvature tensor of $g$. We say that $(\RR^d,g,\mu)$ satisfies the curvature-dimension condition $\mathrm{CD}(\kappa,\infty)$ for some $\kappa > 0$, if for any $x \in \RR^d$,
	$$\mathrm{Ric}_{g,\mu} \succeq \kappa g,$$
	as quadratic forms. For $\kappa>0$, let $\gamma_\kappa^d$
	denote the Gaussian measure with covariance $\kappa^{-1}\mathrm{I}_d$, and let
	$\|\cdot\|$ denote the Euclidean distance. Milman's question asks whether,
	under the condition $\mathrm{CD}(\kappa,\infty)$, there exists a contraction
	\begin{equation} \label{eq:contraction}
		T:(\RR^d,\|\cdot\|,\gamma_\kappa^d)\to(\RR^d,g,\mu).
	\end{equation}
	That is, a map transporting $\gamma_\kappa^d$ to $\mu$ and satisfying
	$d_g(T(x),T(y))\leq \|x-y\|$. As explained in \cite[Section 6]{milman2018spectral} the existence of a map satisfying \eqref{eq:contraction} should be thought of as a tentative question, where a more complete conjecture might also impose separate assumptions on the Ricci curvature. Nevertheless, as we now explain, the curvature-dimension setting can serve as a natural Riemannian analogue for Caffarelli's theorem, and so this is the focus of this work.
	
	It is readily checked that when $g$ is Euclidean, the curvature dimension condition reduces to the usual strong log-concavity condition. Indeed,
	$(\RR^d,\|\cdot\|,\mu)$ satisfies $\mathrm{CD}(1,\infty)$ if and only if $\mu$ is more log-concave than the standard Gaussian. Thus Milman's question can be viewed as a natural Riemannian analogue of Caffarelli's theorem. Moreover, when $d = 1$, the Riemannian metric can be reduced to the Euclidean metric by a change of coordinates, which gives a positive answer to \eqref{eq:contraction} in this case. In contrast, our main result provides a strong counterexample, already in dimension $d = 2$. We construct a weighted Riemannian
	manifold satisfying $\mathrm{CD}(\frac{1}{2},\infty)$ for which no Lipschitz transport map from the corresponding Gaussian source can exist; in particular, there are no contractions.
	
	\begin{theorem}\label{thm:nocontract}
		There exists a Riemannian metric $g$ and a probability measure $\mu$ on $\RR^2$ such that $(\RR^2,g,\mu)$ satisfies $\mathrm{CD}(\frac{1}{2},\infty)$. Moreover, there is no Lipschitz map $T:(\RR^2,\|\cdot\|)\to(\RR^2,g)$ satisfying $T_\#\gamma_{1/2}^2 = \mu$. 
	\end{theorem}
	In our construction, $g$ is a Hessian metric arising from the moment-measure construction, as in \cite{cordero2015moment}, which we review below. The analysis of the curvature-dimension condition is based on recent properties established in \cite{kolesnikov2014hessian} for such Hessian manifolds, while the non-existence of a Lipschitz transport map follows from a geometric argument relating the growth of any such map to the tails of the measures. 
    
    Let us emphasize that Theorem \ref{thm:nocontract} rules out the existence of \emph{any Lipschitz map}, and not only contractions as in \eqref{eq:contraction}. In light of this, the constant $\frac{1}{2}$ is inessential. Indeed, for any $\kappa>0$, a linear map sends
$\gamma^2_{1/2}$ to $\gamma^2_\kappa$ and vice versa. Hence the non-existence of
Lipschitz transports from $\gamma^2_{1/2}$ immediately implies the same
statement for $\gamma_\kappa^2$. Moreover, by
replacing $g$ with $g_\kappa=(2\kappa)^{-1}g$, $(\RR^2,g_\kappa,\mu)$ satisfies $\mathrm{CD}(\kappa,\infty)$, and Theorem \ref{thm:nocontract} applies to this space as well.

  We further remark that many comparable non-existence results for Lipschitz transports from the Gaussian rely on rather straightforward obstructions: either the target measure fails to satisfy a functional inequality that would be inherited under a Lipschitz transport, or the tails of the target are too heavy to be captured by a Lipschitz image of the Gaussian. Here, in contrast, Theorem \ref{thm:nocontract} is a consequence of intrinsic properties of the measured Riemannian space
	
	One motivation for Milman's question comes from a possible comparison principle for eigenvalues of the Laplacian. For a weighted Riemannian manifold $(\RR^d,g,\mu=e^{-F}d\mathrm{Vol}_g)$ define the weighted Laplacian
	$$\Delta_{g,\mu}f:= \Delta_gf -\langle \nabla_g F,\nabla_g f\rangle_g,$$
	and let $\lambda_k(\RR^d,g,\mu)$ stand for the $k^{th}$ eigenvalue of the positive semi-definite operator $-\Delta_{g,\mu}$. By \cite[Theorem 1.7]{milman2018spectral}, if there exists an $L$-Lipschitz $T:(\RR^d,\|\cdot\|,\gamma_\kappa^d)\to(\RR^d,g,\mu)$, then for any $k \in \mathbb{N}$,
	\begin{equation} \label{eq:eigenvaluecomparision}
		\lambda_k(\RR^d,g,\mu) \geq \frac{1}{L^2}\lambda_k(\RR^d,\|\cdot\|,\gamma_\kappa^d).
	\end{equation}
	
	Parallel to the contraction question in \eqref{eq:contraction},
	\cite[Conjecture 1*]{milman2018spectral} asks whether this eigenvalue comparison always holds, with $L = 1$, under the sole assumption that
	$(\RR^d,g,\mu)$ satisfies $\mathrm{CD}(\kappa,\infty)$. A closer analysis of the construction in Theorem \ref{thm:nocontract} gives a counterexample to this spectral question as well. Indeed, for the spaces we construct, the eigenvalues $\lambda_k(\RR^d,\|\cdot\|,\gamma_\kappa^d)$ grow much faster than
	$\lambda_k(\RR^d,g,\mu)$, so for \emph{any} $L$, the inequality in \eqref{eq:eigenvaluecomparision} fails for large enough $k$.
	
	\begin{theorem}\label{thm:nocomparison}
		Let $g$ and $\mu$ be as in Theorem \ref{thm:nocontract}. Then
		$$
		\lambda_k(\RR^2,g,\mu)
		=
		(\sqrt[4]{2}+o(1))\frac{\sqrt{k}}{\log(k)^{1/4}}.
		$$
		Consequently,
		$$
		\lim_{k\to\infty}
		\frac{\lambda_k(\RR^2,g,\mu)}
		{\lambda_k(\RR^2,\|\cdot\|,\gamma_{1/2}^2)}
		=0.
		$$
	\end{theorem}
	In proving Theorem \ref{thm:nocomparison} we prove a Weyl-type asymptotic eigenvalue law for the weighted manifold. The estimates we derive in the proof of Theorem \ref{thm:nocontract} allow us to control the growth of the eigenvalue counting function for $-\Delta_{g,\mu}$, which in turn translates into bounds on $\lambda_k(\RR^2,g,\mu)$. We then compare this with
	the Gaussian case, where the spectrum is explicit and
	$
	\lambda_k(\RR^2,\|\cdot\|,\gamma_{1/2}^2)=\Theta(\sqrt{k}).
	$
	
	Finally, we mention that, by using \eqref{eq:eigenvaluecomparision}, Theorem \ref{thm:nocomparison} implies Theorem \ref{thm:nocontract}. We nevertheless choose to prove Theorem \ref{thm:nocontract} directly, since our argument is more geometric and identifies explicit barriers to the existence of Lipschitz transport maps. As we explained above, this contrasts with many approaches in the literature, where the non-existence of Lipschitz transports is deduced from the failure of functional inequalities such as \eqref{eq:eigenvaluecomparision}.

	\subsection{Ideas of the proof} \label{sec:ideas}
	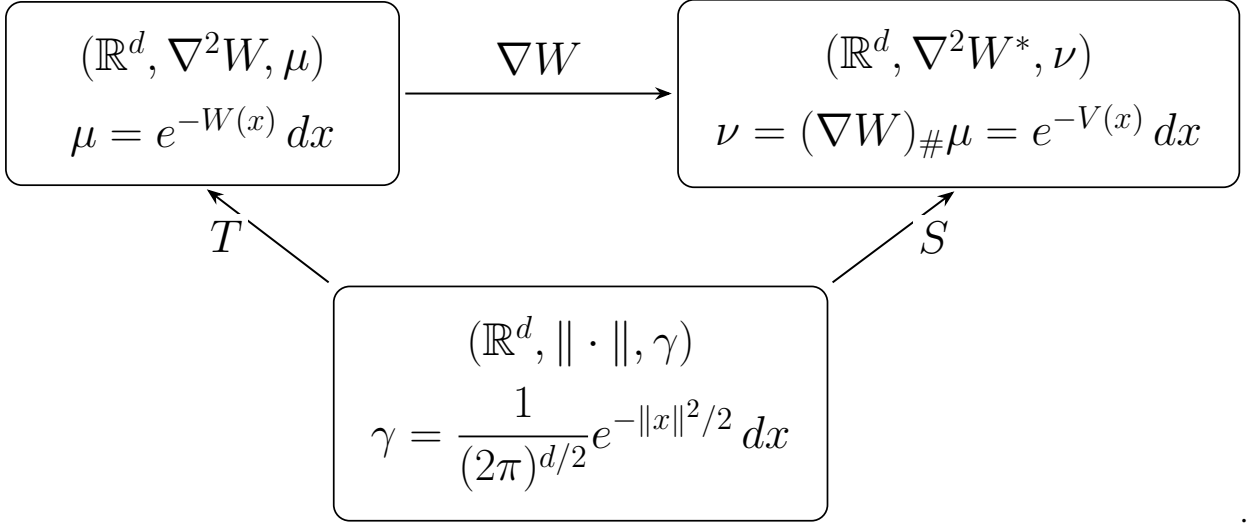
\begin{figure}[h] 
		\caption{Moment measures}\label{fig:Diagram}
		\begin{tikzpicture}[
			>=Stealth,
			every node/.style={align=center},
			box/.style={
				draw,
				rounded corners=6pt,
				line width=0.65pt,
				inner xsep=14pt,
				inner ysep=11pt,
				font=\Large,
				minimum height=1.75cm
			},
			arrow/.style={->, line width=0.75pt, shorten >=2pt, shorten <=2pt},
			lab/.style={font=\Large, fill=white, inner sep=2pt}
			]
			
			\node[box, minimum width=5.2cm] (mu) at (-5.0,4.1)
			{
				$(\mathbb{R}^d,\nabla^2 W,\mu)$\\[7pt]
				$\mu=e^{-W(x)}\,dx$
			};
			
			\node[box, minimum width=5.9cm] (nu) at (5.0,4.1)
			{
				$(\mathbb{R}^d,\nabla^2 W^*,\nu)$\\[7pt]
				$\nu=(\nabla W)_\#\mu=e^{-V(x)}\,dx$
			};
			
			\node[box, minimum width=6.2cm] (gauss) at (0,0)
			{
				$(\mathbb{R}^d,\|\cdot\|,\gamma)$\\[7pt]
				$\displaystyle \gamma=\frac{1}{(2\pi)^{d/2}}e^{-\|x\|^2/2}\,dx$
			};
			
			\draw[arrow] (gauss.north west) -- node[lab, left=7pt, pos=.50] {$T$} (mu.south);
			\draw[arrow] (gauss.north east) -- node[lab, right=7pt, pos=.50] {$S$} (nu.south);
			\draw[arrow] (mu.east) -- node[lab, above=5pt] {$\nabla W$} (nu.west);
			
		\end{tikzpicture}.
	\end{figure}
	As mentioned above, our counterexample in Theorem \ref{thm:nocontract} is
	based on the moment-measure construction. To describe it, let
	$\nu=e^{-V(y)}dy$ be a probability measure on $\mathbb{R}^d$. We say that
	$\nu$ is the moment measure of $\mu$ if
	$\mu=e^{-W(x)}dx$ for some convex function $W$ and
	$$
	(\nabla W)_\#\mu=\nu.
	$$
	Equivalently, by Brenier's theorem \cite{brenier1991polar}, $\nabla W$ is the optimal transport map from $\mu$ to $\nu$.   Under mild assumptions on $\nu$,
	which will be satisfied in our construction, such a measure $\mu$ exists and is unique up to the natural symmetries, see \cite{cordero2015moment}. 
	
	This construction naturally gives rise to two Hessian Riemannian manifolds. Let $W^*$ be the Legendre transform of $W$. Then $\nabla W$ and $\nabla W^*$
	are inverse maps, and $\nabla W^*$ is the optimal transport map from $\nu$ to
	$\mu$. We equip the two spaces with the Hessian metrics induced by this pair of convex functions
	$$
	(\mathbb{R}^d,\nabla^2 W,\mu)
	\qquad\text{and}\qquad
	(\mathbb{R}^d,\nabla^2 W^*,\nu).
	$$
	Differentiating the identity $\nabla W(\nabla W^*(x))=x$ gives
	$\nabla^2 W(\nabla W^*(x))\,\nabla^2 W^*(x)=\mathrm{I}_d,$
	or equivalently
	$$
	\nabla^2 W^*(x)
	=
	\left(\nabla^2 W(\nabla W^*(x))\right)^{-1}.
	$$
	It follows that $\nabla W$ is an isometry from
	$(\mathbb{R}^d,\nabla^2 W,\mu)$ to
	$(\mathbb{R}^d,\nabla^2 W^*,\nu)$, that pushes forward $\mu$ to $\nu$.
	In particular, if there exists a contraction
	$$
	T:(\mathbb{R}^d,\|\cdot\|,\gamma_\kappa^d)
	\to
	(\mathbb{R}^d,\nabla^2 W,\mu),
	$$
	then $S:=\nabla W\circ T$
	is a contraction from
	$(\mathbb{R}^d,\|\cdot\|,\gamma_\kappa^d)$ to
	$(\mathbb{R}^d,\nabla^2 W^*,\nu)$. Thus it is useful to study the two Hessian
	manifolds in parallel. In the construction below, the measure $\nu$ will be chosen explicitly, while $\mu$ is obtained from $\nu$ through the moment-measure construction.

	The curvature-dimension properties of these spaces were studied in \cite{kolesnikov2014hessian} (see also \cite{klartag2014logarithm} and \cite{klartag2017remarks}). The case most relevant to us is when $V$ is convex,
	so that $\nu=e^{-V}dx$ is log-concave. In this setting,
	\cite[Theorem 4.3]{kolesnikov2014hessian} implies that
	$(\mathbb{R}^d,\nabla^2 W,\mu)$ satisfies
	$\mathrm{CD}(1/2,\infty)$, and the same is also true for the isometric space $(\mathbb{R}^d,\nabla^2 W^*,\nu)$.
	
	Our main idea lies in choosing $\nu$ to be a radially symmetric log-concave measure. That is, if $r := \|x\|$ and $\nu = e^{-V(x)}dx$, then $V(x) = v(r)$. By uniqueness in the moment-measure construction, this symmetry is inherited by $\mu$, and $W(x) = w(r)$ as well. Let us focus for now on the \emph{optimal transport maps} from some Gaussian $\gamma_\kappa^d$ to both measures, i.e $T_\#\gamma_\kappa^d = \mu$ and $S_\#\gamma_\kappa^d= \nu$, where both $T$ and $S$ minimize the quadratic. Since $\gamma_\kappa^d$, $\mu$, and $\nu$ are all radially symmetric, this is also true for the optimal transport maps. Thus, we may express them in the form
	$$
	T(x)=t(r)\frac{x}{r},\quad
	S(x)=s(r)\frac{x}{r}, \quad \text{ and }\nabla W(x) = w'(r)\frac{x}{r}.$$
	Moreover, by Brenier's theorem \cite{brenier1991polar}, both $T$ and $S$ are gradients of convex functions. The radial structure also makes these maps compatible with $W$. Indeed, 
	since $t$ and $w'$ are nondecreasing, the map $\nabla W\circ T$ is again the gradient of a convex radial function which also transports $\gamma_\kappa^d$ to $\nu$.
	By uniqueness of the Brenier map from $\gamma_\kappa^d$ to $\nu$, we therefore obtain the identity
	$S=\nabla W\circ T.$
	This setting is summarized in the commutative diagram of Figure \ref{fig:Diagram}.
	
	Now, let us study the Lipschitz properties of $S$ (equivalently of $T$). It is standard to show that as a map from $(\RR^d, \|\cdot\|) \to (\RR^d,\nabla^2 W)$, $T$ is $L$-Lipschitz if and only if, for almost any $x$,
	\begin{equation} \label{eq:riemlip}
		\nabla T(x)^{\top} \cdot \nabla^2 W(T(x)) \cdot \nabla T(x)  \preceq L^2\mathrm{I}_d.
	\end{equation}
	Using the diagram in Figure \ref{fig:Diagram}, together with the
	radial symmetry of all the maps involved, we can compute the eigenvalues of the left-hand side of \eqref{eq:riemlip} explicitly. When $r = \|x\|$, there is one radial eigenvalue given by $t'(r)^2w''(t(r))=t'(r)s'(r)$, and a tangential eigenvalue with multiplicity $d-1$ given by
	$\frac{t(r)s(r)}{r^2}$. The radial eigenvalue is the same quantity that appears in the one-dimensional case, and one can show it is uniformly bounded. In contrast, with a careful choice of $\nu$, we will show that the tangential  eigenvalue can be unbounded, i.e.
	$$\lim\limits_{r\to \infty}\frac{t(r)s(r)}{r^2}=\infty,$$
	ruling out the possibility of $T$ being Lipschitz.
	
	Since we are constraining $\nu$ to be log-concave, it necessarily has at most exponential tails. Compared with the Gaussian tails of $\gamma_\kappa^d$, this suggests that we can choose $\nu$ in a way that forces $s(r)$ to have quadratic growth, i.e $s(r) \geq c\cdot r^2$ for some constant $c$. Indeed, by \cite{colombo2021bounds,fathi2024growth} this is the worst possible growth of $s$. For such a measure we then have
	$$\lim\limits_{r\to \infty}\frac{t(r)s(r)}{r^2}\geq c\cdot\lim\limits_{r\to \infty}t(r)= \frac{c}{2}\mathrm{diameter}\left(\mathrm{support}(\mu)\right).$$
	The equality follows since $T(x) = t(r)\frac{x}{r}$ transports $\gamma_\kappa^d$ unto $\mu$, and so  $t(r)$ must reach the radius of the support of $\mu$. In particular, if $\mathrm{support}(\mu) =  \RR^d$ has infinite diameter, then the tangential eigenvalue is unbounded, and the optimal map is not Lipschitz. Thus, we are looking for a log-concave measure $\nu$, which has \emph{strictly} exponential tails, and which is the moment measure of $\mu$ which has \emph{full support}. 
	
	One possible approach would be to start with a log-concave $\nu = e^{-V(x)}dx$ and choose $V(x)$ so that $V(x) \simeq \|x\|$ when $\|x\|$ is large enough. The technical problem is that we would now need to show that it is the moment measure of $\mu$ and $\mu$ has unbounded support. This would require us to indirectly reason about $\mu$, since explicit solutions to the moment measure problem are notoriously hard to find. In the proof below we will take this route. Here, instead, for the sake of intuition, we draw inspiration from known one-dimensional examples, which are easier to compute in general. Thus, when $d=1$ consider $\nu \propto e^{-|x|}dx$ the Laplace distribution. It is straightforward to verify that it is the moment measure of $\mu \propto e^{-W(x)}dx$, where $W(x) = e^{|x|}-|x|$. Thus $\nu$ has exponential tails and $\mu$ has unbounded support.
	
	In light of the previous discussion, to get a counterexample for the optimal transport map, we need an analog of the one-dimensional example to higher dimensions, when $d \geq 2$. A natural extension to try is $W(x)=d(e^{\|x\|}-\|x\|)$. A computation now shows that if $\mu \propto e^{-W(x)}dx$ and $\nu = \nabla W_{\#}\mu \propto e^{-V(x)}dx$, then 
	$$V(x) = \|x\| -(d-1)\log\left(\frac{d+\|x\|}{\|x\|}\log\left(1+\|x\|/d\right)\right).$$
	Remarkably $V$ remains convex for every $d \in \mathbb{N}$, and clearly it has linear growth. Since $\mu$ has unbounded support, as we explained above, this shows that the optimal transport map from $\gamma$ to $\mu$ cannot be Lipschitz.
	
	
	To adapt these ideas into Theorem \ref{thm:nocontract}, there are two main hurdles to tackle:
	\begin{enumerate}
		\item The densities of $\mu$ and $\nu$ are not smooth, and have a singularity at the origin. Consequently, the same is true for the metrics.
		\item The result only holds for the optimal transport map, and we need to extend it to every transport map.
	\end{enumerate}
	The proof of Theorem \ref{thm:nocontract} will exactly handle those issues. Instead of working with $\nu$ as above, we will work with a smoothed version of the radial Laplace distribution $e^{-\|x\|}dx$. Then we adapt the geometric arguments above to apply for any transport map.

	\subsection{Concurrent and related work}
	A short time before the completion of this work, we became aware of the concurrent work of Aryan \cite{aryan2026obstructions}, which also constructs
	counterexamples to Milman's conjectures. In the setting relevant to \eqref{eq:eigenvaluecomparision}, Aryan constructs, in dimensions $d\geq 4$, a weighted Riemannian manifold $(\RR^d, g,\mu)$ for which \eqref{eq:eigenvaluecomparision} fails when $k = d + 2$ and $L = 1$.
	As we previously explained, this spectral failure also rules out the existence of a contractive transport map, in the sense of \eqref{eq:contraction}.
	
	The two works can be seen as complementary. Aryan's
	construction starts from the spectral side of Milman's question, and proceeds by constructing a space satisfying the curvature-dimension condition for which the eigenvalue comparison in \eqref{eq:eigenvaluecomparision} fails. The main idea in \cite{aryan2026obstructions} is to construct a weighted
	Riemannian manifold whose end becomes effectively cylindrical, which makes it possible to construct several independent test functions with small Rayleigh quotients. With this in mind, the obstruction to transport in \cite{aryan2026obstructions} is then deduced indirectly, from the failure of the eigenvalue comparison. In the present work, the logic is reversed. Our construction is designed from the transport side, using the moment-measure construction and its Hessian geometry. We directly rule out Lipschitz transport maps by showing that the metric growth in our manifold is incompatible with the tails of the source and target measures. The spectral counterexample is then obtained as a consequence of this geometric obstruction. 
	
	Let us point out some further differences between the constructions. First, Theorems \ref{thm:nocontract} and \ref{thm:nocomparison} already apply in dimension $2$. Moreover, Theorem \ref{thm:nocontract} rules out the existence of any Lipschitz transport map, while Theorem \ref{thm:nocomparison} gives an asymptotic failure of the spectral comparison. In contrast \cite{aryan2026obstructions} requires $d\geq 4$ and exhibits a violation of the Gaussian comparison at the fixed finite index $k=d+2$. For transport maps, this then rules out the existence of a contraction, but does not by itself exclude Lipschitz transport maps with larger constants. Indeed, studying the proof of \cite{aryan2026obstructions} suggests that one can construct a transport map which is $L$-Lipschitz for some finite $L$, in which case there is a weakened comparison inequality and \eqref{eq:eigenvaluecomparision} holds with this $L$. On the other hand, the example of \cite{aryan2026obstructions} has an additional geometric feature which our construction does not have, the
	underlying Riemannian metric has non-negative Ricci curvature. Thus, this example also addresses a possible strengthening of Milman's conjecture, see the
	discussion in \cite[Section 6]{milman2018spectral}. As we will show, the Ricci curvature in our example can be negative, see Remark \ref{rmk:curvature}.
	
	In another direction we mention that \cite{aryan2026obstructions} also rules out the spherical version of Milman's question,  \cite[Conjectures 3* and 4*]{milman2018spectral} as well. This should be contrasted with the recent positive results of
	\cite{serres2026contractive,ge2025generalization}, which study contractions from the round sphere to small perturbations of spheres of smaller radius. In particular, these works show that Milman's question has a positive answer for sufficiently small perturbations. However, \cite[Theorem 1]{aryan2026obstructions} shows that such positive results cannot be extended to arbitrary positively curved metrics on the sphere.
	
	Finally we comment on the logarithmic estimate in Theorem \ref{thm:nocomparison} in the context of non-classical
	Weyl laws. Our proof proceeds by computing the
	leading order asymptotics of the eigenvalue counting function, inspired by the technique in \cite[Section 2.3]{milman2018spectral}.This is somewhat reminiscent of
	classical examples such as \cite{simon1983nonclassical, vandenberg1984horn}, where
	logarithmic terms also appear in the principal eigenvalue-counting asymptotics. A recent
	geometric perspective is given by
	\cite{chitour2024weyl}, who construct singular Riemannian structures with prescribed slowly varying factors in Weyl's law, including logarithms.
	
	\paragraph{Acknowledgments:} W.D. was supported by a NSF Graduate Research Fellowship under Grant No. DGE-2140004. He is grateful to Pablo L\'opez Rivera, Jordan Serres, Katharina Eichinger, Max Fathi, Emanuel Milman for helpful conversations, and to Kunal Chawla for pointing out the use of Caffarelli's theorem in the recent resolution of Talagrand's creating convexity conjecture.  D.M. was partially supported by the Brian and Tiffinie Pang Faculty Fellowship. He is grateful to Shrey Aryan for helpful conversations on contractions in Riemannian settings, and to Stefan Steinerberger for enlightening discussions on Weyl's law, as well as to ChatGPT for providing useful references for Proposition \ref{prop:weyl}. 
	\section{Preliminaries}
	\paragraph{Moment measures:}
	Here we give the necessary background on moment measures and their associated Hessian manifolds. Recall that for a given probability measure $\nu = e^{-V(x)}dx$ on $\RR^d$, we are looking for another probability measure $\mu = e^{-W(x)}dx$, such $W$ is convex and $\nabla W_\#\mu = \nu.$
	
	We begin with the necessary regularity assumptions $\nu$ should satisfy to ensure the existence of $\mu$. Here the main definition is that of \emph{essential continuity}. We say that the function $V$ is essentially continuous if it is lower semi-continuous, and the set of points where $V$ is discontinuous has zero $\mathcal{H}^{d-1}$-measure where $\mathcal{H}^{d-1}$ is the $(d-1)$-dimensional Hausdorff measure. In particular, if $V$ is smooth, then it is essentially continuous. The main result of \cite{cordero2015moment} now states that when $V$ is essentially continuous, then $\mu$ is well defined.
	\begin{theorem}[\textnormal{\cite[Theorem 2]{cordero2015moment}}] \label{thm:momentexist}
		Let $\nu = e^{-V(x)}dx$ be a centered probability measure on $\RR^d$, and assume that $V$ is essentially continuous. Then, there exists a convex function $W:\RR^d \to \RR\cup \{\infty\}$ which is essentially continuous such that $\mu = e^{-W(x)}dx$ is a centered probability measure and $\nabla W_\#\mu = \nu$. Moreover, $\mu$ is the unique centered measure with this property.
	\end{theorem}
	
	Next, for a convex function $W$, we present a sufficient condition for the weighted Riemannian manifold $(\RR^d, \nabla^2 W, e^{-W(x)}dx)$ to satisfy the curvature-dimension condition. We follow the presentation of \cite{kolesnikov2014hessian} by introducing transportation metric.
	Let $\mu$ and $\nu$ be two probability measures on $\RR^d$. The optimal transport map aims to find a map $T:\RR^d \to \RR^d$ such that $T_\#\mu = \nu$ which minimizes the squared distance cost (the reader is referred to \cite{villani2003topics} for more details on optimal transport). Brenier's theorem \cite{brenier1991polar} states that under appropriate regularity assumptions such a $T$ exists and moreover $T = \nabla \Phi$, for some convex function $\Phi$, and $T$ is the unique transport map with this property. Now, let $\Psi = \Phi^*$ be the Legendre transform of $\Phi$. As explained in Section \ref{sec:ideas} this leads to two isometric weighted Riemannian manifolds
	$$(\RR^d, \nabla^2 \Phi,\mu) \qquad\text{and}\qquad (\RR^d, \nabla^2 \Psi,\nu),$$
	with $\nabla \Phi$ the isometry between the spaces.
	When $\mu = e^{-W(x)}dx$ and $\nu = e^{-V(x)}dx$, \cite[Theorem 4.3]{kolesnikov2014hessian} gives a sufficient condition for these spaces to satisfy the curvature dimension condition. According to this result, if for every $x \in \RR^d$,
	\begin{equation} \label{eq:kolesnikovsufficnet}
		(\nabla^2\Phi(x))^{-\frac{1}{2}}\cdot \nabla^2W(x)\cdot	(\nabla^2\Phi(x))^{-\frac{1}{2}} + 	(\nabla^2\Phi(x))^{\frac{1}{2}}\cdot \nabla^2V(\nabla \Phi(x))\cdot	(\nabla^2\Phi(x))^{\frac{1}{2}} \succeq 2c\mathrm{I}_d,
	\end{equation}
	then $\mathrm{Ric}_{\nabla^2\Phi, \mu} \succeq c\nabla^2\Phi$. Equivalently
	$(\RR^d, \nabla^2\Phi, \mu)$ satisfies $\mathrm{CD}(c,\infty).$ Using this condition we spell out the curvature-dimension condition when $\Phi = W$.
	\begin{prop} \label{prop:hessianCD}
		Let $\mu = e^{-W(x)}dx$ and $\nu = e^{-V(x)}dx$ be two probability measures on $\RR^d$ with both $W$ and $V$ smooth and convex. Suppose that $\nabla W_\#\mu = \nu$. Then, $(\RR^d, \nabla^2W, \mu)$ satisfies $\mathrm{CD}(\frac{1}{2},\infty).$
	\end{prop}
	\begin{proof}
		It is enough to verify \eqref{eq:kolesnikovsufficnet} with $c = \frac{1}{2}$. Indeed, since $V$ is convex
		$$(\nabla^2\Phi(x))^{\frac{1}{2}}\cdot \nabla^2 V(\nabla W(x))	(\nabla^2\Phi(x))^{\frac{1}{2}}  \succeq 0.$$
		Furthermore, substituting $\Phi$ for $W$ in \eqref{eq:kolesnikovsufficnet} we get
		$$(\nabla^2\Phi(x))^{-\frac{1}{2}}\cdot \nabla^2W(x)	\cdot (\nabla^2\Phi(x))^{-\frac{1}{2}} = \mathrm{I}_d.$$
		Combining the above two displays yields the desired bound.
	\end{proof}
	\paragraph{Contractions on Riemannian manifolds:}
	We record here one important and elementary fact which allows to express the Lipschitz constant of a map between Riemannian manifolds in terms of the Riemannian gradient.
	\begin{lem} \label{lem:liptograd}
		Let $T:\RR^d\to\RR^d$, be continuously differentiable and let $g$ be a Riemannian metric on $\RR^d$. Then the following conditions are equivalent:
		\begin{enumerate}
			\item $T$ is $L$-Lipschitz as map from $(\RR^d,\|\cdot\|)$ to $(\RR^d,g)$. I.e. , for every $x,y\in \RR^d$. $$d_g(T(x),T(y))\leq L\|x-y\|.$$
			\item For almost every $x$ in $\RR^d$, 
			$$\|\nabla T(x)^\top \cdot g(T(x)) \cdot\nabla T(x)\|_{\mathrm{op}} \leq L^2.$$
		\end{enumerate}
		Moreover if $T$ is not continuously differentiable, then the implication $(1) \implies (2)$ still holds.
	\end{lem}
	\begin{proof}
		We first work under the assumption that $T$ is continuously differentiable.
		In that case, we begin by proving that the Lipschitz condition on $T$ implies the operator norm bound on $\nabla T(x)^\top \cdot g(T(x)) \cdot\nabla T(x)$. For $x, v \in \mathbb{R}^d$, $t \in \mathbb{R}$, set $x_t = x+tv$, so that $$d_g(T(x_t), T(x)) \leq L\| x_t-x\| = L\vert t \vert\|v\|.$$ Since $T$ is  $C^1$,
		$$T(x+tv) = T(x)+t\nabla T(x)v+o(t)$$ as $t \rightarrow 0$. We use the following standard first-order expansion of the Riemannian distance in local coordinates, where $p \in \mathbb{R}^d$ is fixed and $h \in \mathbb{R}^d$ is a small displacement vector: $$d_g(p, p+h) = \sqrt{h^\top g(p)h}+o(\| h\|)$$ as $h \rightarrow 0$ which is just the fact that 
		\begin{align*}
			\lim_{h \rightarrow0}\frac{d_g(p, p+h)}{\sqrt{h^\top g(p)h}} = 1.
		\end{align*}
		Thus, using $p = T(x)$ and $h = T(x+tv)-T(x) = t\nabla T(x)v+o(t)$, we have
		\begin{align*}
		 \vert t \vert \sqrt{v^\top \nabla T(x)^\top g(T(x))\nabla T(x)v}+o(\vert t\vert) =	d_g(T(x+tv), T(x)) 
			\leq L\vert t \vert \| v\|.
		\end{align*}
		Dividing by $\vert t \vert$ and sending $t \rightarrow 0$, we get
		\begin{align*}
			\sqrt{v^\top \nabla T(x)^\top g(T(x))\nabla T(x)v} \leq L\| v\|.
		\end{align*}
		
		Conversely, suppose that $\|\nabla T(x)^\top \cdot g(T(x)) \cdot\nabla T(x)\|_{\mathrm{op}} \leq L^2$ for almost all $x$. 
		For any $C^1$ curve $\gamma: [0, 1] \rightarrow \mathbb{R}^d$, its Riemannian length with respect to $g$ is defined by
		\begin{align*}
			\text{Length}_g(\gamma) := \int_0^1 \sqrt{\gamma'(t)^\top g(\gamma(t))\gamma'(t)}\;dt= \int_0^1 \| \gamma'(t)\|_{g, \gamma(t)}\;dt.
		\end{align*}
		Now note that 
		\begin{align*}
			\text{Length}_g(T \circ \gamma) &= \int_0^1 \| \nabla T({\gamma(t)})(\gamma'(t))\|_{g, T(\gamma(t))}\;dt \\
			&= \int_0^1 \sqrt{\gamma'(t)^\top \nabla T(\gamma(t))^\top g(T(\gamma(t)))\nabla T(\gamma(t))\gamma'(t)}\;dt\\
			&\leq \int_0^1\sqrt{L^2\| \gamma'(t)\|^2}\;dt\\
			&= \int_0^1 L \| \gamma'(t)\|\;dt = L \cdot \text{Length}_{\| \cdot \|}(\gamma),
		\end{align*}
		and since this holds for every $C^1$ curve $\gamma$ joining $x$ to $y$, we can take the infimum over all such curves, which yields exactly the Euclidean distance $\| x-y\|$, 
		\begin{align*}
			d_g(T(x), T(y)) &\leq L\inf_{\gamma} \text{Length}_{\| \cdot \|}(\gamma)
			= L\| x-y\|.
		\end{align*}
		
		To finish the proof, consider the case that $T$ is not continuously differentiable. Note that since $g$ is a smooth Riemannian metric, $d_g$ is locally comparable to the Euclidean metric, so the $L$-Lipschitz map $T$ is locally Euclidean Lipschitz, and hence by Rademacher's theorem $T$ being $L$-Lipschitz means that at every point $x \in \RR^2$ where $T$ is differentiable, 
		\begin{align*}
			\nabla T(x)^\top g(T(x))\nabla T(x) \leq L^2 \cdot \mathrm{I}_d
		\end{align*}
	\end{proof}
	\section{Obstructions to Lipschitz transport} \label{sec:nocontract}
	Following the ideas outlined in Section \ref{sec:ideas}, we restrict ourselves to $d=2$ and consider the smooth potential $V(x) = \sqrt{1 + \|x\|^2} - 1$, for which we define the log-concave measure $\nu = \frac{1}{4\pi}e^{-V(x)}dx$. Since $V$ is radial we will also write $V(x) = v(\|x\|)$, where $v:\RR_+\to \RR$ is convex and increasing. 
	\subsection{On the moment measure}
	We begin by studying the moment measure problem for $\nu$.
	By Theorem \ref{thm:momentexist}, there exists a centered probability measure $\mu = e^{-W(x)}dx$ such that $W$ is essentially continuous and convex and $\nabla W_{\#} \mu = \nu$. Moreover, since both $\mu$ is log-concave and $\nu$ has support over $\RR^2$, $W$ is also smooth on its support, see \cite[Theorem 1]{cordero2019regularity}\footnote{The theorem actually says that if both $W$ and $V$ are $k$ times continuously differentiable on their support then $\nabla W$ is $k+1$ times continuously differentiable. Since $W$ is continuous and since $V$ is smooth, we can induct on this claim to show $W$ is actually smooth on its support. }. We first observe that since $\nu$ is radially symmetric the same must be true for $\mu$ as well.
	\begin{lem} \label{lem:rotinvariance}
		Let $\mu$ be as above. There exists a convex and increasing function $w:\RR_+\to \RR \cup \{\infty\}$ such that $W(x) = w(\|x\|)$. Moreover, $w'$ is smooth and invertible on its support, and if we set 
		\begin{equation} \label{eq:qdef}
			q(s) = (w')^{-1}(s)
		\end{equation} to be the inverse function, then we have the identity
		\begin{equation} \label{eq:radialmongeampere}
			e^{-w(q(s))}q(s)q'(s) = \frac{1}{4\pi}e^{-v(s)}s.
		\end{equation}
	\end{lem}
	\begin{proof}
		Let $O \in \mathcal{O}(2)$ be an orthogonal transformation. Define $\mu_O := O_\#\mu$ with density $e^{-W_O(x)}\; dx := e^{-W(O^{-1}x)}\;dx$. Note that 
		\begin{align*}
			\nabla W_O(x) =(O^{-1})^\top\nabla W(O^{-1}x) =O\nabla W(O^{-1}x).
		\end{align*}
		Thus, if $X \sim \mu$, then $OX \sim \mu_O$ and $\nabla W_O(OX) = O\nabla W(X),$ and so
		\begin{align*}
			(\nabla W_O)_\#\mu_O = O_\#(\nabla W)_\#\mu = O_\#\nu.
		\end{align*}
		Since $\nu$ is radial, $O_\#\nu = \nu$. Since $\mu$ is centered, then $\mu_O$ is centered. By uniqueness of the centered moment measure, $\mu_O = \mu$.
		
		Thus, the support of $\mu$ is the centered Euclidean ball $B_R$ of radius $R \in (0, \infty]$, and its density is radial. Then on $B_R$, $W(x) = w(\| x\|)$ for some convex non-decreasing function $w$. The moment map has the radial form $\nabla W(x) = w'(\| x\|) \frac{x}{\|x\|},$ and $\nabla W$ sends the ball $B_{\|x\|}$ onto the ball $B_{w'(\|x\|)}$, and since $\nu$ is the moment measure of $\mu$, \begin{equation}    \label{eq: ballMeasures}
			\mu(B_{\|x\|}) = \nu(B_{w'(\| x\|)}).
		\end{equation}
		
		Since both radial distribution functions in \eqref{eq: ballMeasures} are strictly increasing on their supports, the function $w'$ must be strictly increasing, and $w'(0) = 0$. Since $\nu$ has full support, \eqref{eq: ballMeasures} implies $\lim_{\| x \| \rightarrow R} w'(\|x\|) = \infty$. Thus, $w'$ is invertible from $[0, R)$ onto $\mathbb{R}_+$, and we can define $q(s) = (w')^{-1}(s)$.
		
		Differentiating the radial identity \eqref{eq: ballMeasures}, where $C_\mu, C_\nu$ are the normalizing constants of $\mu, \nu$ respectively,
		\begin{align*}
			\frac{d}{dr}\left(2\pi\int_0^r e^{-w(\rho)}\rho\;d\rho\right) &= \frac{d}{dr}\left(\frac{1}{2}\int_0^{w'(r)} e^{-v(s)}s\;ds\right)
		\end{align*}
		gives
		\begin{align*}
			2\pi e^{-w(r)}r &= \frac{1}{2} e^{-v(w'(r))}w'(r)w''(r).
		\end{align*}

		Substituting $s = w'(r), r = q(s) = (w')^{-1}(s)$ and multiplying by $q'(s) = \frac{1}{w''(q(s))}$ yields the desired identity
		\begin{align*}
			e^{-w(q(s))}q(s)q'(s) &= \frac{1}{4\pi}e^{-v(s)}s.
		\end{align*}
	\end{proof}
	
	We next show that the essential continuity of $W$, as in Theorem \ref{thm:momentexist}, implies that $\mu$ must have unbounded support.
	\begin{lem} \label{prop: Unbounded Support of Mu}
		Let $w:\RR_+ \to \RR \cup \{\infty\}$ be as in Lemma \ref{lem:rotinvariance}. Then,
		\begin{align*}
			R := \sup\{r \geq 0 : w(r) < \infty\} &= \infty.
		\end{align*}
		Equivalently, $\mathrm{support}(\mu) = \RR^2$, and the function $q$ from \eqref{eq:qdef} satisfies $\lim\limits_{s\to\infty}q(s) = \infty$.
	\end{lem}
	\begin{proof}
		Assume, towards contradiction, that $R<\infty$. 
		By construction $\nu$ has full support and since, for $r = \|x\|$, $\nabla W(x) = w'(r)\frac{x}{r}$, the relation $\nabla W_{\#}\mu = \nu$ implies that
		$$\lim\limits_{r\uparrow R} w'(r) = \infty.$$ 
		As in Lemma \ref{lem:rotinvariance} set $q(s):=(w')^{-1}(s)$ for which the above shows
		\begin{equation} \label{eq:qlimit}
			\lim\limits_{s\to \infty} q(s) = R.
		\end{equation}
		Moreover, by Theorem \ref{thm:momentexist}, $W$ is essentially continuous, from which we deduce
		\begin{equation} \label{eq:wlimit}
		 \lim\limits_{r\uparrow R} w(r)= \infty.
		\end{equation}
		Indeed, otherwise $W$ would have a discontinuity on $\partial B_R$ the boundary of the ball of radius $R$, which has positive one-dimensional Hausdorff measure.
		
		Following the change of variable formula in \eqref{eq:radialmongeampere} we set $A(s):=w(q(s)).$
		By definition, $w'(q(s))=s$, and so we have the identity $A'(s)=sq'(s)$. Using these identities we can rewrite \eqref{eq:radialmongeampere} as
		\begin{equation}\label{eq:rewritemongeampere}
			A'(s)=\frac{s^2}{4\pi q(s)}e^{A(s)-v(s)}.
		\end{equation}
		Now, integration by parts gives, for any fixed $s_0>0$,
		$$
		A(s)=A(s_0)+\int_{s_0}^s uq'(u)\,du
		=A(s_0)+s q(s)-s_0q(s_0)-\int_{s_0}^s q(u)\,du.
		$$
		Dividing by $s$ and using $\lim\limits_{u\to \infty} q(u) = R$, we obtain
		$$
		\frac{A(s)}{s}\xrightarrow{s\to \infty} 0.
		$$
		For all sufficiently large $s$, we therefore have $A(s)\leq s/2$ and from \eqref{eq:qlimit}, by making $s$ larger, we can also ensure $q(s)\geq R/2$. 
		Since $v(s)=\sqrt{1+s^2}-1\geq s-1$, \eqref{eq:rewritemongeampere} yields
		$$
		A'(s)\leq \frac{e}{R} s^2 e^{-s/2}
		$$
		for all sufficiently large $s$. The right-hand side is integrable on
		$[1,\infty)$, so $A(s)$ has a finite limit as $s\to\infty$. On the other hand, combining \eqref{eq:qlimit} and \eqref{eq:wlimit} we see that
		$$\lim\limits_{s\to \infty}A(s)=\lim\limits_{s\to \infty}w(q(s))\to\infty.$$ This is a contradiction and so $R=\infty$.
	\end{proof}
	\subsection{Understanding the tails and volume growth of $\mu$}
	We now study the tails and volume growth of $\mu$, the next technical lemma essentially captures the asymptotic behavior of the implicit potential $w$.
	
	\begin{lem} \label{lem : p(s)}
		Let $q(s) = (w')^{-1}(s)$ be as in \eqref{eq:qdef}, and define 
		\begin{align} \label{eq:pdef}
			p(s) = sq'(s).
		\end{align}
		Then $\lim\limits_{s\to \infty}p(s) = 1$.
	\end{lem}
	
	\begin{proof}
		As in the proof of Lemma \ref{lem:rotinvariance}, let $A(s):= w(q(s))$, and note that since $w'(q(s)) = s$, 
		\begin{align*}
			A'(s) &= w'(q(s))q'(s)
			= sq'(s)
			= p(s).
		\end{align*}
		By Lemma \ref{lem:rotinvariance}, we have
		\begin{align*}
			e^{-A(s)}q(s)p(s) &= \frac{1}{4\pi}e^{-v(s)}s^2,
		\end{align*}
		so that
		\begin{equation}\label{eq : A(s)}
			w(q(s)) = A(s) = v(s)+\log(q(s))+\log(p(s))-2\log(s)+\log(4\pi).
		\end{equation}
		Differentiating, and using $A'(s) = p(s)$, we have
		\begin{equation*}
			p(s) = v'(s)+\frac{q'(s)}{q(s)}+\frac{p'(s)}{p(s)}-\frac{2}{s},
		\end{equation*}
		and using $q'(s) = \frac{p(s)}{s}$, we have
		\begin{equation} \label{eq: p(s) as deriv of A(s)}
			\frac{p'(s)}{p(s)}=\left(1-\frac{1}{sq(s)}\right)p(s)-v'(s)+\frac{2}{s}. 
		\end{equation}
		Note that $v(s) = \sqrt{1+s^2}-1$, and so $v'(s) = \frac{s}{\sqrt{1+s^2}},$ and $v'(s) \xrightarrow{s\to \infty} 1$.
		In light of this, fix $\eps \in (0,\frac{1}{10})$, and choose $s_0$ large enough, so that for all $s \geq s_0$,
		\begin{align*}
			v'(s) \geq 1-\frac{\varepsilon}{2} \quad \text{ and } \quad \frac{2}{s} \leq \frac{\eps}{2}.
		\end{align*}
		Suppose that $p(s_0) \leq 1-10\eps$, so that by \eqref{eq: p(s) as deriv of A(s)} we have the differential inequality,
		\begin{align*}
			\frac{p'(s)}{p(s)} \leq p(s) - 1 +\eps, \quad p(s_0)\leq 1-10\eps.
		\end{align*}
		Solving the corresponding differential equation by Gronwall's lemma we deduce that for any $s \geq s_0$,
		$$p(s) \leq \frac{1-\eps}{1+\frac{9\eps}{1-10\eps}e^{(1-\eps)(s-s_0)}}.$$ 
		In particular, since $\eps <\frac{1}{10}$,
		$$\lim\limits_{s\to\infty}q(s) = \int\limits_{0}^\infty q'(r)dr =  \int\limits_{s_0}^\infty \frac{p(r)}{r}dr + q(s_0) < \infty,$$ 
		which contradicts Lemma \ref{prop: Unbounded Support of Mu}. Since the same argument applies to any $s > s_0$ as well, we conclude $p(s) \geq 1-10\eps$ in this regime.
		
		To get an upper bound, we instead use Lemma \ref{prop: Unbounded Support of Mu} according to which $\lim\limits_{s\to\infty}q(s) = \infty$. So we may choose $s_0$ sufficiently large so that
		$$\frac{1}{sq(s)} \leq \eps \text{ and } v'(s) -\frac{2}{s} \leq 1 +\eps.$$
		Thus, assume now that for some $\eps \in (0,\frac{1}{10})$, $p(s_0) \geq 1+10\eps$, in which case \eqref{eq: p(s) as deriv of A(s)} implies the differential inequality
		$$
			\frac{p'(s)}{p(s)} \geq (1-\eps)p(s) - 1 -\eps, \quad p(s_0)\geq 1+10\eps.
		$$
		Proceeding in the same way by solving the corresponding equality, Gronwall's inequality implies this time
		$$p(s) \geq \frac{1+\eps}{1-\eps - \frac{2\eps(4-5\eps)}{1+10\eps}e^{(1+\eps)(s-s_0)}}.$$
		In particular, since $\eps < \frac{1}{10}$, we can see that $p(s)$ blows up at some finite time, contradicting the fact that $q'(s) = \frac{p(s)}{s}$ is finite for $s < \infty$, which follows from Lemma \ref{prop: Unbounded Support of Mu} and the smoothness of $q$, and we conclude that $p(s) \leq 1+10\eps$, when $s$ is large enough.
		
		Combining everything, we have for large $s$, 
		\begin{align*}
			1-10\varepsilon \leq p(s) \leq 1+10\varepsilon,
		\end{align*}
		and since $\varepsilon$ is arbitrary, $\lim\limits_{s\to \infty}p(s) = 1.$
	\end{proof}
	Next, to augment Lemma \ref{prop: Unbounded Support of Mu} we also quantify the rate of divergence for $q$. As we show although $\mu$ has unbounded support its tails decrease at a double exponential rate, consistent with the explicit example from Section \ref{sec:ideas}. This is captured by the fact that $q(s) \simeq \log(s)$, or in other words $w'(s) \simeq e^s$.
	\begin{lem} \label{lem : q(s)}
		Let $q$ be as in \eqref{eq:qdef}. Then,
		\begin{align*}
			q(s) = (1+o(1))\log(s).
		\end{align*}
	\end{lem}
	
	\begin{proof}
		By Lemma \ref{lem : p(s)}, $\lim\limits_{s\to \infty}p(s) = 1$, where $p$ is as in \eqref{eq:pdef} and so 
		\begin{align*}
			q'(s) &= \frac{p(s)}{s} = \frac{1+o(1)}{s}.
		\end{align*}
		Thus, 
		\begin{align*}
			q(s) &= q(1) + \int_{1}^s \frac{p(u)}{u}\;du  = q(1) + (1+o(1))\log(s).
		\end{align*}
	\end{proof}
	
	To prove Theorem \ref{thm:nocontract}, we will need to show that the tails of $\mu$ are too heavy to be compatible with the existence of a Lipschitz transport map when measured against the Riemannian metric induced by $\nabla W$. 
	Towards that, it will be more convenient to rewrite $\mu$ with respect to the Riemannian volume metric. Since $d\mathrm{Vol}_g  = \sqrt{\det{\nabla W^2}}dx$ we have the following expression (see also \cite{kolesnikov2014hessian}),
	\begin{equation} \label{eq:relativevolumedensity}
		\mu \propto e^{-\frac{1}{2}E(x)}d\mathrm{Vol}_g(x), \text{ where } E(x) = W(x) + V(\nabla W(x)).
	\end{equation}
	Using radial symmetry, we also have the expression $E(x) = w(\|x\|) + v(w'(\|x\|)).$
	With respect to $E$ we have the following estimate:
	\begin{lem} \label{lem:volumegrowth}
		There exists a constant $c>0$, such that for any $t >0$ large enough,
		$$\mu(E(x) \geq t) \geq c\sqrt{\log(t)}e^{-\frac{t}{2}}.$$
	\end{lem}
	
	\begin{proof}
		Since all the relevant quantities are radially symmetric, we fix a direction $e_1$ and do all computations in that direction. Thus recall the definition of $q$ from \eqref{eq:qdef}, and define
		\begin{equation} \label{eq : h(s)}
			h(s) := E(q(s)e_1) = w(q(s))+v(s).
		\end{equation}
		If $X \sim \mu$ and $Y = \nabla W(X) \sim \nu$, since we are working with radial measures, we have $\| Y\| = w'(\| X\|)$, and so we have that $E(X) = h(\| Y\|)$, thus 
		\begin{equation*} 
			\mu(E(x) \geq t) = \nu(h(\|y\|) \geq t).
		\end{equation*}
		By adding $v(s)$ to both sides of \eqref{eq : A(s)}, we have
		\begin{equation*}
			h(s) = 2v(s)+\log(q(s))+\log(p(s))-2\log(s)+\log(4\pi).
		\end{equation*}
		Using Lemma \ref{lem : p(s)} and Lemma \ref{lem : q(s)}, along with $v(s) = \sqrt{1+s^2}-1$, we get
		\begin{equation} \label{eq: h(s) modified by estimates}
			h(s) = 2s-2\log(s)+\log\log(s)+O(1).
		\end{equation}
		Furthermore, we can differentiate $h$ to obtain, 
		\begin{align} \label{eq:hderv}
			h'(s) &= (w(q(s))+v(s))'= w'(q(s))q'(s)+v'(s)= p(s)+v'(s)\xrightarrow{s\to \infty} 2,
		\end{align}
		where the last equality follows by definition of $p$, as in \eqref{eq:pdef}. Thus there exists some $s_0 > 0$ so that $h'(s) >0$ when $s > s_0$, and in particular $h$ is invertible on $[s_0,\infty)$. Let $s(t)$ be the inverse function in this regime, so that $h(s(t)) = t$. By \eqref{eq: h(s) modified by estimates}, we have 
		\begin{equation} \label{eq : s(t)}
			s(t) = \frac{t}{2}+\log(t)-\frac{1}{2}\log\log(t)+O(1).
		\end{equation}
		Indeed, if we substitute $s = \frac{t}{2}+\log(t)-\frac{1}{2}\log\log(t)+B$ for some constant $B$ into \eqref{eq: h(s) modified by estimates}, we get $h(s) = t+2B+O(1).$ 
		For large $t$, we then have 
		\begin{align*}
			\mu(E(x) \geq t) &= \nu(h(\|y\|) \geq t) = \nu(\|y\| \geq s(t)).
		\end{align*}
		On the other hand, by definition of $\nu$, we have with a straightforward computation
		\begin{equation} \label{eq : radialTailOfNu}
			\nu(\| y\| \geq s) = \frac{1}{2}\int_s^\infty re^{-\sqrt{1+r^2}+1}dr= \frac{1}{2}(\sqrt{1+s^2}+1)e^{-(\sqrt{1+s^2}-1)} 
		\end{equation}
		Since $s-1<\sqrt{1+s^2-1} < s$, we obtain that 
		$$
		\nu(\|y\|\geq s)\geq \frac{1}{2e}se^{-s}\implies \mu(E(x)\geq t) \geq \frac{1}{2e}s(t)e^{-s(t)}.
		$$ 
		Combining this bound with \eqref{eq : s(t)} now leads to 
		$$\frac{1}{c}\sqrt{\log(t)}e^{-\frac{t}{2}}\geq s(t)e^{-s(t)}\geq c\sqrt{\log(t)}e^{-\frac{t}{2}},$$ where $c>0$ is some constant. This then
		gives the desired bound $\mu(E(x) \geq t) \geq c\sqrt{\log(t)}e^{-\frac{t}{2}}.$

	\end{proof}
	
	\subsection{Volume growth of Lipschitz mappings}
	Now, let us assume the existence of an $L$-Lipschitz map $T$ such that $T_{\#}\gamma^2_{1/2} = \mu$. We show that such an existence implies a bound on the growth of the functional $E$ from \eqref{eq:relativevolumedensity}, in a way which contradicts Lemma \ref{lem:volumegrowth}.
	\begin{lem} \label{lem:contractnogrowth}
		Let $T:(\RR^2,\|\cdot\|) \to (\RR^2,\nabla^2 W)$ and assume that $T_\# \gamma_{1/2}^2=\mu$ and that $T$ is $L$-Lipschitz for some $L > 0$. Then, there exists a constant $C_L\in \RR$, depending only on $L$, such that for any $t$ large enough
		$$\mu(E(x) \geq t + C_L) \leq e^{-\frac{t}{2}}.$$
	\end{lem}
	\begin{proof}
		We will first translate the Lipschitz assumption into a pointwise Jacobian bound. Applying Lemma \ref{lem:liptograd} and since we are working in $\RR^2$, we have the following bound on the Jacobian determinant
		\begin{align} \label{eq:lipbound}
			J_{\nabla^2 W}T(x) &:= J_{\|\cdot\|}T(x)\sqrt{\det\nabla^2W(T(x))} \leq L^2
		\end{align}
		for almost every $x$. Above, $J_{\|\cdot\|}T(x) = \left|\det(\nabla T(x))\right|$ is the standard Euclidean Jacobian determinant. Since $T$ is locally Lipschitz, the area formula applies, and so for almost every $y$, 
		\begin{align*}
			e^{-W(y)} = C_1\sum_{x \in T^{-1}(y)} \frac{e^{-\frac{\| x\|^2}{4}}}{J_{\| \cdot \|}T(x)},
		\end{align*}
		where $C_1>0$ is a normalizing constant.
		The set $D:=\{J_{\|\cdot\|}T(x) = 0\}$ does not cause any difficulties, since by the area formula $T(D)$ has Lebesgue measure $0$, and since $T_\# \gamma_{1/2}^2 = \mu$ with $\mu$ absolutely continuous,  $\gamma_{1/2}^2(D) = 0$. Thus for almost every $x$,
		\begin{align*}
			C_1\frac{e^{-\frac{\| x\|^2}{4}}}{J_{\| \cdot \|}T(x)} \leq  e^{-W(T(x))} \iff 	J_{\| \cdot \|}T(x) &\geq C_1e^{-\frac{\|x\|^2}{4}+W(T(x))}. 
		\end{align*}
		Consider the Monge-Amp\`ere equation for the moment map $\nabla W_\# \mu = \nu$ and take a square root on both sides to obtain
		\begin{align*}
			\sqrt{\det\nabla^2W(x)} &= C_2e^{\frac{1}{2}v(\|\nabla W(x)\|)-\frac{1}{2}W(x)},
		\end{align*}
		where $C_2 > 0$ is another normalizing constant. 
		Combining this identity with the previous inequality we then have
		\begin{align*}
			J_{\nabla^2 W}T(x) &= J_{\| \cdot \|}T(x)\sqrt{\det \nabla^2 W(T(x))}\\
			&\geq  C_1\cdot C_2e^{-\frac{\| x\|^2}{4}+W(T(x))}e^{\frac{1}{2}(v(\|\nabla W(T(x))\|)-W(T(x)))}\\
			&= C_1\cdot C_2e^{-\frac{\|x\|^2}{4}+\frac{1}{2}E(T(x))},
		\end{align*}
		where $E$ is as in \eqref{eq:relativevolumedensity}.
		By the assumption on $T$ \eqref{eq:lipbound},
		\begin{align*}
			L^2 \geq C_1\cdot C_2e^{-\frac{\| x\|^2}{4}+\frac{1}{2}E(T(x))}.
		\end{align*}
		Taking the logarithm, we have
		\begin{align*} 
			-\frac{\| x\|^2}{4}+\frac{1}{2}E(T(x)) \leq \log\left(\frac{L^2}{C_1\cdot C_2}\right).
		\end{align*}
		for almost every $x$. Thus, there exists a constant $C_L\in \RR$ depending only on $L$ such that
		\begin{equation}\label{eq: logEBound}
			E(T(x)) \leq \frac{\| x\|^2}{2}+C_L.
		\end{equation}
	Since $T$ transports $\gamma_{1/2}^2$ to $\mu$
		\begin{align*}
			\mu\left(E(x) \geq t+C_L\right) &= \gamma^2_{1/2}\left(E(T(x)) \geq t+C_L\right),
		\end{align*}
		and from \eqref{eq: logEBound}, we have
		\begin{align*}
			\{E(T(x)) \geq t+C_L\} \subseteq \left\{\frac{\| x\|^2}{2} \geq t\right\},
		\end{align*}
		so
		\begin{align*}
			\mu\left(E(x) \geq t+C_L\right) &\leq \gamma^2_{1/2}\left(\frac{\|x \|^2}{2} \geq t\right)= e^{-\frac{t}{2}}.
		\end{align*}
	\end{proof}
	\subsection{Finishing the proof}
	Combining Lemma \ref{lem:volumegrowth} and Lemma \ref{lem:contractnogrowth} we can now finish the proof.
	\begin{proof}[Proof of Theorem \ref{thm:nocontract}]
		First, since $V$ is convex and smooth, by Proposition \ref{prop:hessianCD}, $(\RR^2, \nabla^2 W,\mu)$ satisfies $\mathrm{CD}(\frac{1}{2},\infty)$. By Lemma \ref{lem:contractnogrowth}, any $L$-Lipschitz transport map $T$ gives an exponential tail for $E$: there exists a constant $C_L \in \RR$ depending only on $L$ such that for all sufficiently large $t$, 
		\begin{equation*}
			\mu(E(x) \geq t+C_L) \leq e^{-\frac{t}{2}}.
		\end{equation*}
		In contrast, by Lemma \ref{lem:volumegrowth}, for all sufficiently large $t$,
		\begin{equation*}
			\mu(E(x) \geq t+C_L) \geq c\sqrt{\log(t+C_L)}e^{-\frac{t+C_L}{2}}.
		\end{equation*}
		Combining, we get
		\begin{equation*}
			c\sqrt{\log(t+C_L)}e^{-\frac{t+C_L}{2}} \leq e^{-\frac{t}{2}} \implies 	c\sqrt{\log(t+C_L)}e^{-\frac{C_L}{2}} \leq 1
		\end{equation*}
		which is impossible when $t\rightarrow \infty$.
	\end{proof}

	\section{Analyzing the eigenvalue counting function}
	With the estimates established in Section \ref{sec:nocontract} we can now also study the eigenvalues of the Laplacian $\Delta_{g,\mu}$. For the purposes of proving Theorem \ref{thm:nocomparison} it would suffice to bound the eigenvalues from above. Instead, we will completely characterize their growth by proving an asymptotic Weyl law. The main technical difficulty lies in the fact that our manifold is not compact, and so some care is required to verify that the asymptotic law holds.
	
	We will take the same route as in \cite[Section 2.3]{milman2018spectral} and study the eigenvalue counting function of $M = (\RR^2, g = \nabla^2 W, \mu = e^{-W}dx = e^{-F}d\mathrm{Vol}_g)$,
	\begin{align*}
		N_M(\lambda) &:= \#\{k : \lambda_k(M) \leq \lambda\},
	\end{align*}
	where $W$ is as in Section \ref{sec:nocontract} and where the relative density $F=\frac{E}{2}$, up to a suppressed additive constant, with $E$ defined in \eqref{eq:relativevolumedensity}. 
	The first step in \cite[Section 2.3]{milman2018spectral} is to apply the unitary transformation $L^2(\mu) \rightarrow L^2(d\mathrm{Vol}_g), f \mapsto e^{-F/2}f$. Under this multiplier the weighted Laplacian
	\begin{align*}
		-\Delta_{g, \mu} &= -\Delta_g +\langle \nabla_g F, \nabla_g \cdot\rangle_g,
	\end{align*}
	becomes the Schr\"odinger operator 
	\begin{equation*}
		H = -\Delta_g+\mathcal{V},
	\end{equation*}
	where 
	\begin{equation} \label{eq: SchrodingerOperatorTerm}
		\mathcal{V} = \frac{1}{4}\| \nabla_g F\|_g^2-\frac{1}{2}\Delta_g F.
	\end{equation}
	Crucially, unitary transformations preserve eigenvalues and so $H$ and $-\Delta_{g, \mu}$ have the same spectrum.
	
	Our proof of Theorem \ref{thm:nocomparison} now relies on the following non-compact Weyl law. To prove this Weyl law holds we rely on some recent results providing sufficient conditions, \cite{dai2025weyl}. Verifying these conditions in practice requires some technical care and so we delay the proof to the end of the section.
	\begin{prop} \label{prop:weyl}
		 It holds that 
		\begin{equation} \label{eq: WeylLaw}
		N_M(\lambda) = (1+o(1)) \frac{1}{4\pi}\int_M (\lambda - \mathcal{V}(x))_+d\mathrm{Vol}_g(x).
		\end{equation}
	\end{prop}
	
	With the Weyl law at hand, we can prove our main estimate.
	
	\begin{lem} \label{lem : CountingLowerBound}
		For all sufficiently large $\lambda$,
		\begin{equation*} 
			N_M(\lambda) =  (1+o(1))\lambda^2\sqrt{\log(\lambda)}.
		\end{equation*}
	\end{lem}
	
	\begin{proof}
	 Recall the two functions $q(s) = (w')^{-1}(s)$ and $p(s) = sq'(s)$ from Section \ref{sec:nocontract}. Using these functions we rewrite the Riemannian metric $g = \nabla^2 W$ in radial coordinates $(s,\theta)$
		\begin{align} \label{eq : Hessian Metric in Radial Coordinates}
			g &= w''(r)\;dr^2+rw'(r)\;d\theta^2 = \frac{p(s)}{s}\;ds^2+sq(s)\;d\theta^2,
		\end{align}
	where we write $r = q(s)$, so that $dr = q'(s)\;ds = \frac{p(s)}{s}\;ds$ and using the fact $w''(r) = \frac{1}{q'(s)} = \frac{s}{p(s)}$. Since now we have $\det g = p(s)q(s)$, the Riemannian volume element can be written as
		\begin{equation} \label{eq: RiemannianVolumeElementInRadial}
			d\mathrm{Vol}_g = \sqrt{p(s)q(s)}\;ds\;d\theta.
		\end{equation}
		Fix a direction $e_1$, and as in \eqref{eq : h(s)} define
		$$f(s):=F(q(s)e_1) = \frac{1}{2}(w(q(s))+v(s)),$$
		and as in \eqref{eq:hderv} we have $f'(s) = \frac{1}{2}(p(s)+v'(s)) \xrightarrow{s\to \infty} 1.$
		
		
		We now estimate \eqref{eq: SchrodingerOperatorTerm} by combining the facts that $F$ is radial, that $f'(s) \xrightarrow{s\to \infty} 1$, and that $g^{ss} = \frac{s}{p(s)}$, where $g^{ss}$ is the component of $g^{-1}$,  and that according to Lemma \ref{lem : p(s)} $\lim\limits_{s \to \infty}p(s) = 1.$
		Putting these things together implies for $x \in \RR^2$ with $\|x\| = r$,
		\begin{align} \label{eq:lingradient}
			\| \nabla_g F(x)\|_g^2 &= g^{ss}(f'(s))^2 
			= \frac{s}{p(s)}(f'(s))^2
			= (1+o(1))s,
		\end{align}
		for some constant $C > 0$.

		We also need to evaluate the Laplacian term. Thus, when $\|x\| = r$, we have 
		$$\Delta_gF(x) = 	 \frac{1}{\sqrt{p(s)q(s)}}\frac{d}{ds}\left(s\sqrt{\frac{q(s)}{p(s)}}f'(s)\right)=  \frac{1}{\sqrt{p(s)q(s)}}\frac{d}{ds}\left(s\sqrt{\frac{q(s)}{p(s)}}\cdot \frac{1}{2}(p(s)+v'(s))\right).$$
		We now evaluate the derivative, so that
		\begin{align*}
			\Delta_g F 
			&= \frac{1}{2\sqrt{p(s)q(s)}}\left(\sqrt{\frac{q(s)}{p(s)}}(p(s)+v'(s))+s\left(\sqrt{\frac{q(s)}{p(s)}}\right)'(p(s)+v'(s))+s\sqrt{\frac{q(s)}{p(s)}}(p'(s)+v''(s))\right)\\
			&= \frac{1}{2\sqrt{p(s)q(s)}}\sqrt{\frac{q(s)}{p(s)}}\left( (p(s)+v'(s))+\frac{s}{2}(p(s)+v'(s))\left(\frac{q'(s)}{q(s)}-\frac{p'(s)}{p(s)}\right)+ s(p'(s)+v''(s))\right)\\
			&= \frac{1}{2p(s)}\left( p(s)+v'(s)+s(p'(s)+v''(s))+\frac{s}{2}(p(s)+v'(s))\left(\frac{q'(s)}{q(s)}-\frac{p'(s)}{p(s)}\right)\right)
		\end{align*}
		Apply now the identity $q'(s) =\frac{p(s)}{s}$, so that 
		\begin{align*}
			\Delta_g F &= \frac{1}{2p(s)}\left( p(s)+v'(s)+s(p'(s)+v''(s))+\frac{s}{2}(p(s)+v'(s))\left(\frac{p(s)}{sq(s)}-\frac{p'(s)}{p(s)}\right)\right)\\
			&= \frac{p(s)+v'(s)}{2p(s)}+\frac{sp'(s) + sv''(s)}{2p(s)}+\frac{p(s)+v'(s)}{4q(s)}-\frac{s(p(s)+v'(s))p'(s)}{4p(s)^2}\\
			&= \frac{p(s)+v'(s)}{2p(s)}+\frac{ sv''(s)}{2p(s)}+\frac{p(s)+v'(s)}{4q(s)}+\frac{s(p(s)-v'(s))p'(s)}{4p(s)^2}\\
			&= \frac{p(s)+v'(s)}{2p(s)}+\frac{ sv''(s)}{2p(s)}+\frac{p(s)+v'(s)}{4q(s)}+\frac{s(p(s)-v'(s))}{4p(s)}\left(p(s)-v'(s)-\frac{p(s)}{sq(s)}+\frac{2}{s}\right),
			\end{align*}
			where the last identity is \eqref{eq: p(s) as deriv of A(s)}. By expanding the last term we get
			\begin{align*}
			\Delta_g F&= \frac{p(s)+v'(s)}{2p(s)}+\frac{sv''(s)}{2p(s)}+\frac{p(s)+v'(s)}{4q(s)}+\frac{s(p(s)-v'(s))^2}{4p(s)}-\frac{p(s)-v'(s)}{4q(s)}+\frac{p(s)-v'(s)}{2p(s)}\\
			&= \frac{2p(s)+v'(s)-v'(s)}{2p(s)}+\frac{2v'(s)}{4q(s)}+\frac{sv''(s)}{2p(s)}+\frac{s(p(s)-v'(s))^2}{4p(s)}\\
			&= 1+\frac{v'(s)}{2q(s)}+\frac{sv''(s)}{2p(s)}+\frac{s(p(s)-v'(s))^2}{4p(s)} = o(s),
		\end{align*}
		where we've again used Lemma \ref{lem : p(s)} and Lemma \ref{lem : q(s)} for the final bound, along with $v'(s) = (\sqrt{1+s^2})' = 1 + o(1)$ and $v''(s) = (\sqrt{1+s^2})'' = O\left(\frac{1}{s^3}\right)$.
		Thus, along with \eqref{eq:lingradient} we see
		\begin{align} \label{eq:Vexp}
			\mathcal{V}(s) = \frac{1}{4}\| \nabla_g F\|_g^2-\frac{1}{2}\Delta_gF
			= (1+o(1))\frac{s}{4}.
		\end{align}
		By Proposition \ref{prop:weyl}, and by moving to radial coordinates, we now have,
		\begin{align*}
N_M(\lambda) = \frac{1+o(1)}{4\pi}\int_M (\lambda - \mathcal{V}(x))_+d\mathrm{Vol}_g(x)	 &= \frac{1+o(1)}{2}\int_0^\infty (\lambda - \mathcal{V}(s))_+\sqrt{p(s)q(s)}ds\\
		&= \frac{1+o(1)}{2}\int_0^\infty (\lambda - \mathcal{V}(s))_+\sqrt{\log(s)}ds\\
		&= (1+o(1))\lambda^2\sqrt{\log(\lambda)}
		\end{align*}
	The second identity is \eqref{eq: RiemannianVolumeElementInRadial}, the third follows from Lemmas \ref{lem : p(s)} and \ref{lem : q(s)}, and the final bound follows from the estimates on $\mathcal{V}$ in \eqref{eq:Vexp} and integration.
	\end{proof}
	Theorem \ref{thm:nocomparison} now follows in a straightforward manner from Lemma \ref{lem : CountingLowerBound}.
	
	\begin{proof}[Proof of Theorem \ref{thm:nocomparison}]
	Inverting the relation
		\begin{align*}
			N_M(\lambda)= (1+o(1))\lambda^{2}\sqrt{\log(\lambda)}
		\end{align*}
	from Lemma \ref{lem : CountingLowerBound} shows that for $k$ large enough
		\begin{align*}
		\lambda_k(\RR^2,g,\mu)  =(\sqrt[4]{2}+o(1))\frac{\sqrt{k}}{\log(k)^{1/4}}.
		\end{align*}
		Meanwhile, the eigenvalues of the Gaussian are explicit and satisfy (see \cite{milman2018spectral} for example),
		\begin{align*}
			\lambda_k(\RR^2, \| \cdot \|,\gamma_{1/2}^2) &= \Theta\left(\sqrt{k}\right).
		\end{align*}
		Combining the estimates gives
		\begin{align*}
			\lim_{k\to\infty}
			\frac{\lambda_k(\RR^2,g,\mu)}
			{\lambda_k(\RR^2,\|\cdot\|,\gamma^2_{1/2})}
			&=0.
		\end{align*}
		
	\end{proof}
\subsection{Establishing the asymptotic Weyl law}
To finish the proof it remains to prove Proposition \ref{prop:weyl}. 
	\begin{proof}[Proof of Proposition \ref{prop:weyl}]
		The proof of the Weyl law follows by verifying the conditions in \cite[Theorem 1.10]{dai2025weyl}. This requires us to both prove that $(\RR^2,\nabla^2W)$ has bounded geometry and to prove that the potential $\mathcal{V}$ satisfies appropriate growth and regularity conditions.
		\paragraph{Completeness:} We first claim that $(\RR^2, \nabla^2 W)$ is a complete Riemannian manifold. Indeed as in \eqref{eq: RiemannianVolumeElementInRadial}, $g = \nabla^2 W$ has two eigenvalues, $w''(r) = \frac{s}{p(s)}$ and $\frac{w'(r)}{r} = \frac{s}{q(s)}$. By Lemmas \ref{lem : p(s)} and \ref{lem : q(s)} we see that both eigenvalues diverge and so $\nabla^2 W \succeq a\mathrm{I}_d$, for some $a>0$. Now for any absolutely continuous curve $\gamma$,
		$$
		\mathrm{Length}_g(\gamma)
	\geq
	\sqrt{a}\,\mathrm{Length}_{\|\cdot\|}(\gamma).
		$$
		Equivalently,
		$$
		d_g(x,y)\ge \sqrt{a}\,\|x-y\|.
		$$
		Therefore if $x_n$ is Cauchy sequence with respect to $d_g$, it is also Cauchy with respect to the Euclidean metric, and hence $\|x_n - x\|\to 0$, for some
		$x\in\RR^2$. The
		Riemannian and Euclidean distances are locally comparable near $x$, so $d_g(x_n,x)\to 0$ as well. Hence $(\RR^2,\nabla^2W)$ is complete.
		
		\paragraph{Bounded curvature:}
		We next prove that $(\RR^2,g)$ has bounded curvature as in \cite[Definition 1.2(2)]{dai2025weyl}. We thus need to bound the Riemannian curvature tensor, together with its first covariant derivative. For that we shall need to control the higher derivatives of $p$. Recall from \eqref{eq: p(s) as deriv of A(s)} that
		$$
		\frac{p'(s)}{p(s)}
		=
		\left(1-\frac1{sq(s)}\right)p(s)-v'(s)+\frac2s.
		$$
		Below we shall repeatedly use the estimates $p(s) = 1 +o(1)$ and $q(s) = (1+o(1))\log(s)$ as in Lemmas \ref{lem : p(s)} and \ref{lem : q(s)}. Set
		$
		m(s):=
		\frac{v'(s)-2/s}{1-\frac1{s q(s)}},
		$
		so the above becomes
		\begin{equation} \label{eq:pderv}
		p'(s)
		=
		p(s)\left(1-\frac1{s q(s)}\right)(p(s)-m(s)).
		\end{equation}
		Since $v'(s)= (\sqrt{1+s^2})' = 1+O\left(\frac{1}{s^2}\right)$, we have
		$$
		m(s)
		=
		1-\frac2s+\frac1{s q(s)}+O\left(\frac{1}{s^2}\right)
		\text{ and }
		m'(s)=O\left(\frac{1}{s^2}\right).
		$$
		Moreover $m(s)\xrightarrow{s\to\infty}1$, and hence $p(s)-m(s)\xrightarrow{s\to\infty}0$. We now write a differential equality for $p-m$, 
		$$
		(p(s)-m(s))'
		=
		p(s)\left(1-\frac1{s q(s)}\right)(p(s)-m(s))-m'(s).
		$$
		Solving this equation with the boundary condition $\lim\limits_{s\to\infty}\left(p(t) - m(t)\right) = 0$ gives,
		$$
		p(s)-m(s)
		=
		\int_s^\infty
		m'(u)\exp\left(
		-\int_s^u p(t)\left(1-\frac1{tq(t)}\right)\,dt
		\right)\,du = O\left(\frac{1}{s^2}\right),
		$$
		where we used $m'(u)=O\left(\frac{1}{u^2}\right)$. Therefore
		$$
		p(s)
		=
		1-\frac2s+\frac1{s q(s)}+O\left(\frac{1}{s^2}\right).
		$$
		In particular, going back to \eqref{eq: h(s) modified by estimates},
		$$
		p'(s)=O\left(\frac{1}{s^2}\right).
		$$
		Differentiating \eqref{eq:pderv} once more, and applying the above bounds along with $q'(s) = \frac{p(s)}{s} = O\left(\frac{1}{s}\right)$ also gives
		$$
		p''(s)=O\left(\frac{1}{s^2}\right).
		$$
		Consequently,
		\[
		q''(s)= \frac{p'(s)}{s} - \frac{p(s)}{s^2} = O\left(\frac{1}{s^2}\right).
		\]
		We use these bounds to bound the curvature tensor. Since we are working in dimension $2$ it is enough to understand the Gauss curvature. For a rotationally invariant metric
		$$
		g=F(s)ds^2+G(s)d\theta^2,
		$$
		the Gauss curvature is
		$$
		K(s)
		=
		-\frac1{2\sqrt{F(s)G(s)}}
		\frac{d}{ds}
		\left(
		\frac{G'(s)}{\sqrt{F(s)G(s)}}
		\right).
		$$
		In our case, as in \eqref{eq : Hessian Metric in Radial Coordinates},
		$$
		F(s)=\frac{p(s)}{s}
		\quad \text{ and } \quad
		G(s)=s q(s).
		$$
		So,
		\begin{equation} \label{eq:gausscurv}
		K(s)
		=
		-\frac{1}{2\sqrt{p(s)q(s)}}
		\frac{d}{ds}
		\left(
		\frac{q(s)+p(s)}{\sqrt{p(s)q(s)}}
		\right)=:	-\frac{1}{2\sqrt{p(s)q(s)}}
		\frac{d}{ds}H(s).
		\end{equation}
		The above bounds on $p,q$ and their derivatives imply
		$$
		|H'(s)|=O\left(\frac{1}{s\sqrt{q(s)}}\right),
		$$
		which leads to the uniform bound on the Gauss curvature
		$$
		|K(s)|
		=
		\frac{|H'(s)|}{2\sqrt{p(s)q(s)}}
		=
		O\left(\frac{1}{s q(s)}\right) = O(1).
		$$
		We also need a bound on the first covariant derivative of the curvature tensor, which again reduces to $\|\nabla_g K\|_g$.
		Taking a second derivative of $H$ and invoking the bounds on $p''$ and $q''$ now gives 
		$$
		|H''(s)|=O\left(\frac{\sqrt{q(s)}}{s^2}\right).
		$$	
		Differentiating the expression for $K$ we can conclude with,
		$$
		K'(s)
		=
		-\frac{H''(s)}{2\sqrt{p(s)q(s)}}
		+
		\frac{H'(s)(p(s)q(s))'}{4(p(s)q(s))^{3/2}} = O\left(\frac{1}{s^2}\right)
		$$
		Finally, since $K$ is radial and $g^{ss}=\frac{s}{p(s)}$, we have
		$$
		|\nabla_gK|_g^2
		=
		g^{ss}(K'(s))^2
		=
		\frac{s}{p(s)}(K'(s))^2
		=
		O\left(\frac{1}{s^3}\right).
		$$ Hence the curvature tensor and its first covariant derivative are uniformly
		bounded, which is the curvature part of the bounded geometry condition in \cite[Definition 1.2(2)]{dai2025weyl}.
		\paragraph{Injectivity radius:}
		Next we bound the injectivity radius from below. 
		For this we shall use the results of Cheeger, Gromov, and Taylor \cite[Theorem 4.7]{cheeger1982finite}, according to which, and in light of the bounded curvature of the previous step, it is enough to show that there exist
		$r_0,v_0>0$ such that
		\begin{equation} \label{eq:gromov}
		\mathrm{Vol}_g(B_g(x,r_0))\geq v_0
	\text{ for all }x\in\RR^2,
		\end{equation}
		where $B_g(x,r_0)$ is the $d_g$-ball of radius $r_0$ around $x$.
		
		Introduce the geodesic radial coordinate
		\[
		\rho(s):=\int_0^s\sqrt{\frac{p(u)}u}\,du .
		\]
		Since $p(s)\xrightarrow{s\to \infty}1$, we have $\rho(s)= (1+o(1))2\sqrt{s}$. Using $\rho$ for a change of variables transforms the metric \eqref{eq : Hessian Metric in Radial Coordinates} into
		$$
		g=d\rho^2+a(\rho)^2\,d\theta^2,
		$$
		where
		$a(\rho(s))=\sqrt{s q(s)}.$
		Moreover, since $q'(s) = \frac{p(s)}{s}$,
		$$
		\frac{d}{d\rho}\log(a(\rho))=\frac{ds}{d\rho}\frac{d}{ds}\log(a(\rho(s))
		=
		\sqrt{\frac{s}{p(s)}}\frac{d}{ds}\log\sqrt{s q(s)}
		=
		\frac{1}{2\sqrt{p(s)s}}
		\left(1+\frac{p(s)}{q(s)}\right)
		=
		O\left(\frac{1}{\rho(s)}\right).
		$$
		Therefore, since the last term is $o(1)$, for all sufficiently large \(\rho_0\),
		$$
		|\rho-\rho_0| \leq \frac{1}{C'}\implies \frac{1}{C'}\leq \frac{a(\rho)}{a(\rho_0)} \leq C',
		$$
		for some numeric constant $C'>0$.
		Fix $x=(\rho_0,\theta_0)$ with $\rho_0$ large, and set
		$$
		u=\rho-\rho_0 
		\quad \text{ and }\quad
		y=a(\rho_0)(\theta-\theta_0).
		$$
		In these coordinates,
		$$
		g
		=
		du^2+
		\left(\frac{a(\rho_0+u)}{a(\rho_0)}\right)^2dy^2.
		$$
		Hence, on a fixed rectangle
		$Q=\{|u|\le c,\ |y|\le c\}$,
		with $c>0$ sufficiently small, the metric is comparable to the Euclidean metric and the volume form is comparable to $dudy$. Choosing $c$ smaller if necessary,
		we have
		$
		Q\subset B_g(x,r_0)
		$
		for some fixed $r_0>0$, independent of $x$. Consequently,
		$$
		\mathrm{Vol}_g(B_g(x,r_0))\ge \mathrm{Vol}_g(Q)\geq v_0
		$$
		for all $x$ outside a compact set, with $v_0>0$ independent of $x$. Since the metric is positive definite, this, together with a compactness argument, establishes \eqref{eq:gromov} on the entire $\RR^2$, and through \cite[Theorem 4.7]{cheeger1982finite} we obtain a lower bound on the injectivity radius. Coupled with the previous steps we proved that $(\RR^2,\nabla^2W)$ satisfies the bounded geometry assumption of \cite[Definition 1.2]{dai2025weyl}.
		\paragraph{Diverging potential:} From \eqref{eq:Vexp} 
		and since we've used the change of coordinates $ r = q(s)$,
		
		$$\mathcal{V}(re_1) = \mathcal{V}(q(s)e_1)=(1+o(1))\frac{s}{4}.
		$$
		Since $s=w'(r)\xrightarrow{r\to\infty}\infty$, we get $$\mathcal{V}(x)\xrightarrow{\|x\|\to\infty}\infty.$$
		\paragraph{Doubling condition:}
		Suppose $\lambda$ is large and let
		\begin{align*}
			\sigma(\lambda):=\mathrm{Vol}_g\{x:\mathcal V(x)\le \lambda\} = \int_M {\bf 1}_{\mathcal{V}\leq \lambda} d\mathrm{Vol}_g(x) \geq (1+o(1))2\pi \int_0^{(4-\delta)\lambda}\sqrt{p(s)q(s)}ds,
		\end{align*}
		where $\delta$ is some small fixed number, and where we have again used \eqref{eq:Vexp} as well as the expression for the Riemannian volume \eqref{eq: RiemannianVolumeElementInRadial}. Applying Lemmas \ref{lem : p(s)} and \ref{lem : q(s)} we deduce
		$$
		\sigma(\lambda)
		=
		\Omega(\lambda\sqrt{\log\lambda}).
		$$
		We can apply an analogous argument for an upper bound, which then implies the doubling condition \(\sigma(2\lambda)\leq C\sigma(\lambda)\) for large \(\lambda\), and for some $C>0$.
		
		\paragraph{Regularity condition:}
		Finally we verify the regularity condition of \cite[Definition 1.4]{dai2025weyl}.
		Adding a constant to $\mathcal V$ only shifts the spectrum, and does not affect the
		asymptotics of \eqref{eq: WeylLaw}. Thus, because of \eqref{eq:Vexp} we may assume $\mathcal V \geq 1$.
		Combining the expressions from \eqref{eq:lingradient} and \eqref{eq:Vexp} we have
		$$
		\mathcal V(s)
		=
		\frac{1}{4}\frac{s}{p(s)}
		\left(\frac{p(s)+v'(s)}{2}\right)^2
		-
		\frac{1}{2}
		\left(
		1+\frac{v'(s)}{2q(s)}
		+\frac{s v''(s)}{2p(s)}
		+\frac{s(p(s)-v'(s))^2}{4p(s)}
		\right).
		$$
		Moreover, recall the estimates from the previous steps
		$$
		p(s)=1-\frac{2}{s}+\frac{1}{s q(s)}+O\left(\frac{1}{s^2}\right),
		\quad
		p'(s)=O\left(\frac{1}{s^2}\right),
		\quad
		q(s)= (1+o(1))\log s,
		\quad
		q'(s)=\frac{p(s)}s,
		$$
		and combine  with
		$$
		v = \sqrt{1+s^2}-1,\quad v'(s)=1+O\left(\frac{1}{s^2}\right),\quad
		v''(s)=O\left(\frac{1}{s^3}\right),\quad
		v'''(s)=O\left(\frac{1}{s^4}\right),
		$$
		to get,
		$$
		\mathcal V'(s)=\frac{1}{4}+O\left(\frac{1}{s}\right).
		$$
		Since $g^{ss}=\frac{s}{p(s)}$, this implies through \eqref{eq:Vexp},
		$$
		\|\nabla_g \mathcal V\|_g^2
		=
		g^{ss}(\mathcal V'(s))^2
		=
		\frac{s}{p(s)}(\mathcal V'(s))^2
		\leq C s
		\leq 10C \mathcal V(s).
		$$
		Let $x,y\in\RR^2$ with $d_g(x,y)<\alpha$, where $\alpha>0$ is any fixed number. Since
		$$
		\|\nabla_g \sqrt{\mathcal V}\|_g
		=
		\frac{\|\nabla_g \mathcal V\|_g}{2\sqrt{ \mathcal V}}
		\leq 5C,
		$$
		we have
		$$
		|\sqrt{\mathcal V(x)}-\sqrt{\mathcal V(y)}|
		\leq 5C d_g(x,y)
		\leq 5C\alpha.
		$$
		Moreover,
		$\sqrt{\mathcal{V}(y)}\leq \sqrt{\mathcal{V}(x)}+C\alpha.$
		Since \(V\ge1\), this gives 
		$$\sqrt{\mathcal V(x)}+\sqrt{\mathcal V(y)} \leq (C\alpha+2)\sqrt{\mathcal V(x)}.$$
		Therefore
		\begin{align*}
		|\mathcal{V}(x)-\mathcal{V}(y)|
		=
		|\sqrt{\mathcal{V}(x)}-\sqrt{\mathcal{V}(y)}|
		\left(\sqrt{\mathcal{V}(x)}+\sqrt{\mathcal{V}(y)}\right)  \leq
		5C\alpha(C\alpha+2)\sqrt{\mathcal{V}(x)}.
		\end{align*}
		Equivalently,
		$$
		|\mathcal{V}(x)-\mathcal{V}(y)|
		\leq
		\mathcal{V}(x)\eta(\mathcal{V}(x)),
		$$
		where $\eta(t):=5C\alpha(C\alpha+2)t^{-1/2}$.
		Since $\mathcal{V} \geq 1$, we can modify $\eta$ to be
		continuous and strictly decreasing, and so it satisfies the regularity assumption of \cite[Definition 1.4]{dai2025weyl} with $\beta = 0$.
		\paragraph{The Weyl law for $M$:}
		After verifying all the conditions above, \cite[Theorem 1.10]{dai2025weyl} gives the asymptotic law in \eqref{eq: WeylLaw} for the Schr\"odinger operator $H$ from \eqref{eq: SchrodingerOperatorTerm}. Since $H$ is unitarily equivalent to $-\Delta_{g,\mu}$ the proof concludes.
	\end{proof}
	\begin{remark} \label{rmk:curvature}
		Going back to the expression for the Gauss curvature in \eqref{eq:gausscurv}, we can use the established bounds on $p,q$ and their derivatives to show that $K$ must be negative for some points. Since we are in dimension $2$, then the same is true for the Ricci curvature. In fact, this is a common occurrence for such transportation manifolds, see \cite{klartag2017remarks}.
	\end{remark}
	\bibliographystyle{plain}
	\bibliography{bib}{}

\begin{thebibliography}{10}

\bibitem{aryan2026obstructions}
Shrey Aryan.
\newblock Spectral obstructions to contracting transport maps on curved spaces.
\newblock {\em arXiv preprint arXiv:2605.24705}, 2026.

\bibitem{brenier1991polar}
Yann Brenier.
\newblock Polar factorization and monotone rearrangement of vector-valued functions.
\newblock {\em Comm. Pure Appl. Math.}, 44(4):375--417, 1991.

\bibitem{brigati2025heat}
Giovanni Brigati and Francesco Pedrotti.
\newblock Heat flow, log-concavity, and {L}ipschitz transport maps.
\newblock {\em Electron. Commun. Probab.}, 30:Paper No. 71, 12, 2025.

\bibitem{caffarelli2000monotonicity}
Luis~A. Caffarelli.
\newblock Monotonicity properties of optimal transportation and the {FKG} and related inequalities.
\newblock {\em Comm. Math. Phys.}, 214(3):547--563, 2000.

\bibitem{carlier2026optimal}
Guillaume Carlier, Alessio Figalli, and Filippo Santambrogio.
\newblock On optimal transport maps between {$\frac{1}{d}$}-concave densities.
\newblock {\em Ann. Inst. H. Poincar\'e{} C Anal. Non Lin\'eaire}, 43(2):483--500, 2026.

\bibitem{cheeger1982finite}
Jeff Cheeger, Mikhail Gromov, and Michael Taylor.
\newblock Finite propagation speed, kernel estimates for functions of the {L}aplace operator, and the geometry of complete {R}iemannian manifolds.
\newblock {\em J. Differential Geometry}, 17(1):15--53, 1982.

\bibitem{chewi2024entropic}
Sinho Chewi and Aram-Alexandre Pooladian.
\newblock An entropic generalization of {C}affarelli's contraction theorem via covariance inequalities.
\newblock {\em C. R. Math. Acad. Sci. Paris}, 361:1471--1482, 2023.

\bibitem{chitour2024weyl}
Y.~Chitour, D.~Prandi, and L.~Rizzi.
\newblock Weyl's law for singular {R}iemannian manifolds.
\newblock {\em J. Math. Pures Appl. (9)}, 181:113--151, 2024.

\bibitem{colombo2021bounds}
Maria Colombo and Max Fathi.
\newblock Bounds on optimal transport maps onto log-concave measures.
\newblock {\em J. Differential Equations}, 271:1007--1022, 2021.

\bibitem{colombo2017lipschitz}
Maria Colombo, Alessio Figalli, and Yash Jhaveri.
\newblock Lipschitz changes of variables between perturbations of log-concave measures.
\newblock {\em Ann. Sc. Norm. Super. Pisa Cl. Sci. (5)}, 17(4):1491--1519, 2017.

\bibitem{cordero2015moment}
D.~Cordero-Erausquin and B.~Klartag.
\newblock Moment measures.
\newblock {\em J. Funct. Anal.}, 268(12):3834--3866, 2015.

\bibitem{cordero2019regularity}
Dario Cordero-Erausquin and Alessio Figalli.
\newblock Regularity of monotone transport maps between unbounded domains.
\newblock {\em Discrete Contin. Dyn. Syst.}, 39(12):7101--7112, 2019.

\bibitem{dai2025weyl}
Xianzhe Dai and Junrong Yan.
\newblock Weyl {L}aw for {S}chr\"odinger {O}perators on {N}oncompact {M}anifolds, {H}eat {K}ernel, and {K}aramata-{H}ardy-{L}ittlewood {T}heorem.
\newblock {\em arXiv preprint arXiv:2504.15551}, 2025.

\bibitem{fathi2024growth}
Max Fathi.
\newblock Growth estimates on optimal transport maps via concentration inequalities.
\newblock {\em arXiv preprint arXiv:2407.11951}, 2024.

\bibitem{fathi2020proof}
Max Fathi, Nathael Gozlan, and Maxime Prod'homme.
\newblock A proof of the {C}affarelli contraction theorem via entropic regularization.
\newblock {\em Calc. Var. Partial Differential Equations}, 59(3):Paper No. 96, 18, 2020.

\bibitem{fathi2024transportation}
Max Fathi, Dan Mikulincer, and Yair Shenfeld.
\newblock Transportation onto log-{L}ipschitz perturbations.
\newblock {\em Calc. Var. Partial Differential Equations}, 63(3):Paper No. 61, 25, 2024.

\bibitem{ge2025generalization}
Yuxin Ge and Jordan Serres.
\newblock A generalization of caffarelli's contraction theorem to nearly spherical manifolds.
\newblock {\em arXiv preprint arXiv:2512.01496}, 2025.

\bibitem{HuaSongTudose2026}
Dongming~Merrick Hua, Antoine Song, and Stefan Tudose.
\newblock On {T}alagrand's {C}onvexity {C}onjecture.
\newblock {\em arXiv preprint arXiv:2605.10908}, 2026.

\bibitem{kim2012generalization}
Young-Heon Kim and Emanuel Milman.
\newblock A generalization of {C}affarelli's contraction theorem via (reverse) heat flow.
\newblock {\em Math. Ann.}, 354(3):827--862, 2012.

\bibitem{klartag2017remarks}
B.~Klartag and A.~V. Kolesnikov.
\newblock Remarks on curvature in the transportation metric.
\newblock {\em Anal. Math.}, 43(1):67--88, 2017.

\bibitem{klartag2014logarithm}
Bo'az Klartag.
\newblock Logarithmically-concave moment measures {I}.
\newblock In {\em Geometric aspects of functional analysis}, volume 2116 of {\em Lecture Notes in Math.}, pages 231--260. Springer, Cham, 2014.

\bibitem{kolesnikov2013sobolev}
A.~V. Kolesnikov.
\newblock On {S}obolev regularity of mass transport and transportation inequalities.
\newblock {\em Theory Probab. Appl.}, 57(2):243--264, 2013.

\bibitem{kolesnikov2011mass}
Alexander~V Kolesnikov.
\newblock Mass transportation and contractions.
\newblock {\em arXiv preprint arXiv:1103.1479}, 2011.

\bibitem{kolesnikov2014hessian}
Alexander~V. Kolesnikov.
\newblock Hessian metrics, {$CD(K,N)$}-spaces, and optimal transportation of log-concave measures.
\newblock {\em Discrete Contin. Dyn. Syst.}, 34(4):1511--1532, 2014.

\bibitem{lopez2025bakry}
Pablo L\'opez-Rivera.
\newblock A {B}akry-\'emery approach to {L}ipschitz transportation on manifolds.
\newblock {\em Potential Anal.}, 62(2):331--353, 2025.

\bibitem{mikulincer2023lipshitz}
Dan Mikulincer and Yair Shenfeld.
\newblock On the {L}ipschitz properties of transportation along heat flows.
\newblock In {\em Geometric aspects of functional analysis}, volume 2327 of {\em Lecture Notes in Math.}, pages 269--290. Springer, Cham, [2023] \copyright 2023.

\bibitem{milman2018spectral}
Emanuel Milman.
\newblock Spectral estimates, contractions and hypercontractivity.
\newblock {\em J. Spectr. Theory}, 8(2):669--714, 2018.

\bibitem{neeman2022lipschitz}
Joe Neeman.
\newblock Lipschitz changes of variables via heat flow.
\newblock {\em arXiv preprint arXiv:2201.03403}, 2022.

\bibitem{serres2026contractive}
Jordan Serres.
\newblock Contractive transport maps from {$\Bbb{S}^{2}$} to nearly spherical surfaces with positive {R}icci curvature.
\newblock {\em Nonlinear Anal.}, 267:Paper No. 114058, 8, 2026.

\bibitem{shenfeld2024exact}
Yair Shenfeld.
\newblock Exact renormalization groups and transportation of measures.
\newblock {\em Ann. Henri Poincar\'e}, 25(3):1897--1910, 2024.

\bibitem{simon1983nonclassical}
Barry Simon.
\newblock Nonclassical eigenvalue asymptotics.
\newblock {\em J. Funct. Anal.}, 53(1):84--98, 1983.

\bibitem{song2026}
Antoine Song.
\newblock Sum of {G}aussian vectors and large sets.
\newblock {\em arXiv preprint arXiv:2602.22342}, 2026.

\bibitem{vandenberg1984horn}
M.~van~den Berg.
\newblock On the spectrum of the {D}irichlet {L}aplacian for horn-shaped regions in {${\bf R}\sp{n}$}\ with infinite volume.
\newblock {\em J. Funct. Anal.}, 58(2):150--156, 1984.

\bibitem{villani2003topics}
C\'{e}dric Villani.
\newblock {\em Topics in optimal transportation}, volume~58 of {\em Graduate Studies in Mathematics}.
\newblock American Mathematical Society, Providence, RI, 2003.

\end{thebibliography}
	
\end{document}